\documentclass[11pt]{amsart}

\usepackage{amssymb}
\usepackage{amsmath}
\usepackage{amsfonts}
\usepackage{graphicx}
\usepackage{amsthm}
\usepackage{enumerate}
\usepackage[mathscr]{eucal}
\usepackage{mathrsfs}
\usepackage{verbatim}
\usepackage{xcolor}
\usepackage{hyperref}
\usepackage{wrapfig}
\usepackage{color}

\makeatletter
\@namedef{subjclassname@2010}{%
	\textup{2010} Mathematics Subject Classification}
\makeatother

\numberwithin{equation}{section}

\theoremstyle{plain}
\newtheorem{theorem}{Theorem}[section]
\newtheorem{lemma}[theorem]{Lemma}
\newtheorem{proposition}[theorem]{Proposition}

\newtheorem{defn}[theorem]{Definition}

\theoremstyle{plain}

\numberwithin{equation}{section}

\theoremstyle{remark}

\DeclareMathOperator{\vol}{vol}

\DeclareMathOperator{\rank}{Rank}
\DeclareMathOperator{\Span}{Span}
\DeclareMathOperator{\spec}{Spec}
\DeclareMathOperator{\inte}{Int}
\DeclareMathOperator{\Length}{Length}

\begin{document}
	\date{\today}
	
	\title[Prescription of the Robin spectrum]
	{Prescription of the Robin spectrum}
	
	\author{Xiang He}
	\address{Yau Mathematical Sciences Center\\
		Tsinghua University\\
		Beijing, 100084\\ P.R. China\\}
	\email{x-he@mail.tsinghua.edu.cn}
	
	\author{Zuoqin Wang}
	\address{School of Mathematical Sciences\\
		University of Science and Technology of China\\
		Hefei, 230026\\ P.R. China\\}
	\email{wangzuoq@ustc.edu.cn}
	
	\begin{abstract}
	Let $M$ be a  compact connected smooth manifold with smooth boundary, and let $\rho$ be a positive continuous function on the boundary which is served as the Robin parameter. In this paper, we study three problems concerning the prescription of finite Robin spectrum:
		\begin{enumerate}
			\item Prescribing finitely many Robin eigenvalues and the volume.
			\item Within a fixed conformal class, prescribing the multiplicities of finitely many Robin eigenvalues.
			\item Within a fixed conformal class, prescribing finitely many distinct Robin eigenvalues and the volume.
		\end{enumerate}
		A key step in our method is to solve the corresponding problems for the Dirichlet spectrum. As a consequence, we also obtain analogous results for the Dirichlet case.
	\end{abstract}
	
	\maketitle
	
	\section{Introduction}
	
	Let $M$ be a compact connected smooth manifold of dimension $n$. For closed manifolds ($\partial M = \emptyset$) with $n\ge 3$, Y. Colin de Verdi\`ere's seminal work (\cite{CV}) demonstrated that given any finite non-decreasing sequence of positive numbers \[0=a_1<a_2 \le a_3 \le \cdots \le a_m,\] 
	there is a Riemannian metric $g$ on $M$ realizing this sequence as the first $m$ eigenvalues of the Laplace-Beltrami operator,
	i.e., $\lambda_i(M, g)= a_i$ for $1 \le i \le m$. J. Lohkamp (\cite{JL96}) later extended this result significantly: beyond spectral prescription, the Riemannian metric $g$ can be chosen to simultaneously prescribe both the volume $\mathrm{Vol}(M, g)$ and  the total scalar curvature $\int_M R_g dV_g$, where $R_g$ denotes the scalar curvature of $(M, g)$.  
	
	For compact manifold $M$ with non-empty smooth boundary $\partial M$, the Laplace-Beltrami operator is typically studied under two canonical boundary conditions: the  Dirichlet boundary condition and the Neumann boundary condition. While Colin de Verdi\`ere's method \cite{CV} adapts naturally  to the Neumann case and yielding  analogous  spectral prescription results for the first $m$ Neumann eigenvalues,  its direct application to Dirichlet boundary conditions faces significant obstructions. In our recent work \cite{HW} we resolves this limitation by developing some new techniques, establishing that for any $m \in \mathbb N$, there exists Riemannian metrics $g$ on $M$ that realize the arbitrary prescribed first $m$ Dirichlet eigenvalues and has the pre-assigned volume simultaneously. The problem of prescribing eigenvalues for various other operators has also been studied. P. Guerini considered the Hodge–de Rham operator in \cite{Gu04}; M. Dahl studied the Dirac operator in \cite{Da05}; and P. Jammes investigated a range of operators, including the Hodge–de Rham operator, the Witten Laplacian, and the Steklov operator, in \cite{PJcon,PJ09,PJJGA,PJ,PJ12,PJSte}.

	There also many other natural boundary conditions for the Laplacian that has been studied extensively, among which the Robin boundary condition occupies a distinctive position due to its hybrid nature --- interpolating between Dirichlet and Neumann conditions via a spectral parameter. This condition, critical in modeling heat transfer phenomena with convective surface interactions, introduces unique analytical challenges for spectral prescription. In this work, we address the following question: Can one simultaneously prescribe the Robin eigenvalues, geometric quantities, and boundary interaction parameters for a compact Riemannian manifold?
	
	Before stating our results, we fix some notations:
	\begin{itemize}
		\item  Let $\mathrm{d} V_g$ be the  volume form on $(M, g)$, and $\mathrm{d} \sigma_g$  the induced volume form on $\partial M$. 
		\item Let $\rho \in C_+(\partial M)=\{\rho\in C(\partial M)|\ \rho(x)>0,\ \forall x\in \partial M\}$ be the Robin parameter (which is a function instead of a constant throughout this paper) governing boundary interaction strength.
		\item Let 	
		\[
		0< \lambda_1(M, g,\rho) < \lambda_2(M, g,\rho) \leq \cdots \to \infty
		\]
		be the Robin spectrum of the triple $(M, g, \rho)$, namely the spectrum of the Laplace-Beltrami operator  under  the Robin boundary condition $\partial_\nu f+\rho f=0$, where $\nu$ is the unit external normal vector on $\partial M$. 
	\end{itemize}  
	According to the standard spectral theory, the Robin Laplacian  $\Delta(M,g,\rho)$ can be realized as the Friedrichs extension of the following quadratic form on $H^1(M)$,
	\begin{equation}
		q_\rho (f)= \int_M |\nabla_g f|^2 \mathrm{d} V_g+ \int_{\partial M} \rho \cdot f^2 \mathrm{d} \sigma_g.
	\end{equation}
	Note that we always have  $\lambda_1(M, g,\rho) < \lambda_2(M, g,\rho) $ since we took $M$ to be connected. 
	
	Our first theorem is an extension of the main result in \cite[Theorem 1.2]{HW} and establishes a triple prescription result in the Robin setting: one can  simultaneously prescribe the Robin parameter, the first $m$ Robin eigenvalues and the volume $\mathrm{Vol}(M, g)$:
	\begin{theorem}\label{prevalueRobin}
		Let $M$ be a compact connected smooth $n$-manifold  ($n \ge 3$) with non-empty boundary. For any prescribed data
		\begin{itemize}
			\item Ascending finite  sequence    $0<a_1<a_2\leq \cdots \leq a_m$,
			\item Boundary Robin parameter $\rho\in C_+(\partial M)$,
			\item Target volume $V>0$,
		\end{itemize}
		there exists a Riemannian metric $g$ on $M$ such that $\mathrm{Vol}(M, g)=V$ and $\lambda_k(M, g,\rho)=a_k$ for all $1\leq k\leq m$.
	\end{theorem}
    The strategy of the proof is to reduce the Robin eigenvalue prescription to the resolved Dirichlet case. This is accomplished in two steps. First, following the strategy of \cite[Theorem 3.6]{PJSte}, we  demonstrate that it suffices to establish the result for $\rho \ge \rho_0$, where $\rho_0$ is a sufficiently large lower bound. Second, we analyze  the asymptotic relation   
    \[\lambda_k(M,g,\rho) \stackrel{\rho \to +\infty}{\longrightarrow}\lambda_k^\mathcal{D}(M,g),\] 
    where $\lambda_k^\mathcal{D}(M,g)$ is the $k$-th eigenvalue of the Dirichlet Laplacian, to  bridge  our previous  Dirichlet prescription results \cite{HW} to the current Robin framework.

    In contrast to higher-dimensional manifolds, surfaces exhibit intrinsic spectral limitations due to topological constraints. Let $X$ be a compact surface with $b \ge 1$ boundary components and Euler characteristic $\chi(X)$. When $\chi(X)+b <0$, Berdnikov established in  \cite{Be18} the following multiplicity bound,
    \[\mathrm{mult}(\lambda_k(X)) \le 2k -2(\chi(X)+b)+1, \quad \forall k \ge 1.\]
	This inequality imposes a fundamental restriction: arbitrary finite spectral prescription becomes impossible on such surfaces. 
	However, by following  the same strategy used in the proof of Theorem \ref{prevalueRobin}, we are still able to  extend   the main result in \cite[Theorem 1.1]{H} to the Robin case and  obtain the following partial prescription result:
	
	\begin{theorem}\label{2dimRobin}
		Let $X$ be a compact connected surface with non-empty boundary.  For any prescribed data
		\begin{itemize}
			\item Strictly increasing finite  sequence    $0<a_1<a_2< \cdots < a_m$
			\item Boundary Robin parameter $\rho\in C_+(\partial X)$
			\item Target area $A>0$
		\end{itemize}
		there exists a Riemannian metric $g$ on $X$ such that $\mathrm{Area}(X, g)=A$ and $\lambda_k(X, g,\rho)=a_k$ for all $1\leq k\leq m$. 
	\end{theorem}

	For the rest of this paper, we turn to eigenvalue prescription within fixed conformal classes. The analogous problems for the Hodge–de Rham Laplacian on closed manifolds and for Steklov eigenvalues were first studied by P. Jammes in \cite{PJcon} and \cite{PJSte}, respectively.
	
	So we fix a compact connected Riemannian $n$-manifold ($n \ge 3$) with boundary, $(M, g)$, and a positive boundary Robin parameter $\rho \in C_+(\partial M)$. Define the {\sl distinct} Robin eigenvalues (excluding multiplicities) as:
	\[
	0< \Lambda_1(M,g, \rho) < \Lambda_2(M,g, \rho) < \Lambda_3(M,g, \rho) < \cdots \to \infty,
	\]
	Our next result shows that the multiplicity of Robin eigenvalues can be prescribed by Riemannian metrics within given conformal class:	
	\begin{theorem}\label{premulRobin}
		Let $M$ be a compact connected smooth $n$-manifold  ($n \ge 3$) with non-empty boundary. For any 
		\begin{itemize}
			\item Multiplicity sequence $\{m_1,\cdots,m_s\}$ in $\mathbb{N}$ 
			\item Boundary Robin parameter $\rho\in C_+(\partial M)$ 
			\item Conformal class $[g]$
		\end{itemize}
		there exists a Riemannian metric   $\tilde g$ inside the conformal class $[g]$ such that   the multiplicity of $\Lambda_{k+1}(M,\tilde g,\rho)$ is $m_k$ for all $1\leq k\leq s$.
	\end{theorem}
	
	Again we reduce the Robin eigenvalue multiplicity prescription problem to the analogous Dirichlet eigenvalue multiplicity prescription problem. We remark that the Dirichlet eigenvalue multiplicity prescription problem was first proposed in \cite{HKP} as an open problem, and was answered in a stronger form  in our earlier work  \cite{HW}. It seems to us that the problem ``prescribing the multiplicity within the given conformal class" is more natural, which we will answer in this paper. 
	
	In fact, when $M$ is a closed manifold with $\dim M \ge 3$, Y. Colin de Verdière has already solved the problem of prescribing the multiplicity of the first nonzero Laplacian eigenvalue in \cite{CV2}. Moreover, the argument in \cite{CV2} shows that one can prescribe the multiplicities of finitely many eigenvalues.
		
	To complete the proof of Theorem \ref{premulRobin}, we also follow the strategy of \cite{CV2}, with suitable modifications to address Dirichlet eigenvalues and to incorporate the method from \cite{PJSte}, which ensures that the metric remains in a fixed conformal class. More precisely, inspired by \cite{CV2} we construct a quantum graph $(G_N^{\gamma^0},\Delta_{\gamma^0})$ with a stable metric $\gamma^0$ and high  second eigenvalue multiplicity. Then for any metric $\gamma$ near $\gamma^0$, we construct a corresponding metric $g_\gamma$ on $M$ such that the spectral difference between the Dirichlet Laplacian of $(M, g_\gamma)$ and $(G_N^{\gamma}, \Delta_\gamma)$ is very small. Finally one may utilize the stability of $\gamma^0$ to ensure the  existence of the desired metric on $M$. The challenge here is that the standard quantum graph approach used in \cite{CV2} does not inherently fix the conformal class. To solve this problem we adopt the technique from the proof of \cite[Theorem 1.2]{PJSte}.

    It is well known that for the Neumann eigenvalue (or the closed eigenvalue is the manifold is closed) $\mu_k(M, g)$, one cannot simultaneously prescribe the eigenvalue and the volume inside given conformal class, since  $\mu_k(M,g)\cdot \vol(M,g)^{\frac 2n}$  has an upper bound for metrics within any fixed conformal class; see \cite{LY}, \cite{EI}, \cite{Kor93}, \cite{Has}. 
    However, our last result in this paper reveals a fundamental distinction between the spectral behaviors of the Robin/Dirichlet boundary conditions versus the Neumann/closed cases. 
	\begin{theorem}\label{prediffeigen}
			Let $M$ be a compact connected smooth $n$-manifold  ($n \ge 3$) with non-empty boundary. For any prescribed data
			\begin{itemize}
				\item  Strictly increasing finite  sequence   $0<a_1<a_2< \cdots < a_m$
				\item Boundary Robin parameter $\rho\in C_+(\partial M)$
				\item Target volume $V>0$
				\item Conformal class $[g]$
			\end{itemize}
			there exists a Riemannian metric  $\tilde g$ in $[g]$ such that $\mathrm{Vol}(M, \tilde g)=V$ and $\lambda_k(M, \tilde g,\rho)=a_k$ for all $1\leq k\leq m$.
	\end{theorem}
	
	Again we convert  the problem  to the case of Dirichlet boundary condition, i.e. the following theorem:
	\begin{theorem}
		Let $M$ be a compact connected smooth $n$-manifold  ($n \ge 3$) with non-empty boundary. For any prescribed data
		\begin{itemize}
			\item  Strictly increasing finite  sequence   $0<a_1<a_2< \cdots < a_m$
			\item Target volume $V>0$
			\item Conformal class $[g]$
		\end{itemize}
		there exists a Riemannian metric  $\tilde g$ in $[g]$ such that $\mathrm{Vol}(M, \tilde g)=V$ and $\lambda_k^\mathcal{D}(M, \tilde g)=a_k$ for all $1\leq k\leq m$, where $\lambda_k^\mathcal{D}(M, \tilde g)$ is the k-th Dirichlet eigenvalue of $(M,\tilde g)$.
	\end{theorem} 
	\noindent The proof is built upon the ideas  established in Theorem \ref{premulRobin}, together with two new techniques. First, instead of   prescribing the multiplicity of eigenvalues, we prescribe $m$ distinct eigenvalues by embedding $m$ quantum graphs geometrically into $M$. Second, we simultaneously prescribe the volume by appropriately scaling the metric. This permits us to identify a domain $\Omega$ in $M$ with volume $V$ and whose first Laplacian eigenvalue under a specific mixed boundary condition is significantly large. Moreover, to apply the standard machine of spectral convergence, it is necessary for $\Omega$ to be spatially disjoint from the image of the quantum graphs embedded in  $M$.

We end with two remarks on the limitation of our last result: 
		\begin{enumerate}
			\item   Our method in proving Theorem \ref{prediffeigen}  does not apply to  surfaces ($n=2$), since we need \cite[Lemma 5.1]{HW} which  requires $n\geq3$
			\item  The no-multiplicity assumption in Theorem \ref{prediffeigen}  cannot be removed, since we don't have a
 universal quantum graph model that can replicate arbitrary finite eigenvalue sequences as initial  part of its eigenvalues. 
		\end{enumerate}

	The paper is arranged as follows. In Section \ref{Qgraph} we construct the quantum graph $(G_N^\gamma, \Delta_\gamma)$ and establish several key technical results that will be used in the subsequent proof.  Section \ref{proof1} contains complete proofs of Theorems \ref{prevalueRobin} and \ref{2dimRobin}.
 Section \ref{proof2} presents the proof of Theorems \ref{premulRobin} and \ref{prediffeigen}. The Appendix \ref{app} provides foundational material on the theory of spectral convergence and stable metric.

	\section{The quantum graph $(G_N^\gamma, \Delta_\gamma)$}\label{Qgraph}
	
	In this section, we construct a quantum graph whose second eigenvalue has a large multiplicity. Moreover, under perturbations of the edge length parameters, the eigenspace associated with the second eigenvalue splits but remains stable in the sense of the quadratic form.
	
	For any positive integer $N$, let $G_N=(V,E)$ be the graph formed by adding $N$ boundary vertices to the complete graph on $N$ interior vertices, with each boundary vertex connected to a distinct interior vertex. More precisely,  the vertex set of $G_N$ is 
	\[V=\{v_1,\dots,v_N\}\cup\{u_1,\dots,u_N\},\]
	and the edge set of $G_N$ is
	\[
	E=\{(v_i,v_j),1\leq i<j\leq N\}\cup\{(u_k,v_k),1\leq k\leq N\}.
	\]
	However, unlike the combinatorial framework in  \cite{HW},  we will regard $G_N$ as a $1$-complex (whose edges are parameterized on the interval $[0,1]$) and in what follows we will equipped $G_N$  with a quantum graph structure.
	
	Let  $\gamma\in C(G_N)$ be a \textbf{metric element} on $G_N$, in other words, $\gamma$ is smooth and strictly positive when restricted  to each edge in $E$.  For simplicity we denote
	\[
	\gamma_{ij}(t)=\gamma|_{(v_i,v_j)}(t),\qquad \gamma_k(t)=\gamma|_{(u_k,v_k)}(t)
	\]
	and  
	\[
	l_{ij}= \int_0^1 \gamma_{ij}(t) \mathrm{d}t, \qquad l_k= \int_0^1 \gamma_k(t) \mathrm{d}t.
	\]
    Let $G^{\gamma}_N$ be the metric graph constructed by assigning length  $l_{ij}$ to each edge $(v_i,v_j)$ and the length $l_k$ to each edge $(u_k,v_k)$. As usual we let $L^2(G_N^{\gamma})$ be the space of $L^2$-integrable functions on $G_N^{\gamma}$ and  $H_0^1(G_N^{\gamma})$ be the space of functions  $f\in C(G_N^{\gamma})$   satisfying  conditions
    \begin{enumerate}
    	\item 
    	      For each boundary vertex $u_k$, $f(u_k)=0$;
    	\item 
    	      For each edge $(v_i,v_j)$ (parametrized as $[0, l_{ij}]$), \[f_{ij}:=f|_{(v_i,v_j)}\in H^1([0,l_{ij}]);\]
    	\item 
    	For each boundary edge $(u_k,v_k)$ (parametrized as $[0, l_{k}]$), \[f_k:=f|_{(u_k,v_k)}\in H^1([0,l_k]).\]
    \end{enumerate}
    The Dirichlet Laplacian $\Delta_\gamma$ on $G^\gamma_N$ is defined as the Friedrichs extension of the following quadratic form  on $H_0^1(G_N^{\gamma})$,
	\[
	q_\gamma(f)= \sum_{k=1}^N \int_0^{l_k} |f_k'(t)|^2 \mathrm{d}t+ \sum_{1\leq i<j\leq N} \int_0^{l_{ij}} |f_{ij}'(t)|^2 \mathrm{d}t.
	\]
	The domain of $\Delta_\gamma$ consists of functions $f\in H_0^1(G_N^{\gamma})$ such that 
	\begin{enumerate}
		\item The restriction of $f$ to each edge belongs to $H^2$.
		\item At each interior vertex $v_i$, the sum of the derivatives of $f$ along all incident edges equals zero.
	\end{enumerate}
	
	Let $\gamma^0$ be the simplest metric element, defined by $\gamma^0=1$ on each edge of $G_N$. The first $N$ eigenvalues and eigenfunctions of $\Delta_{\gamma^0}$ can be calculated explicitly:
	
	\begin{lemma}\label{spec of gamma0}
		If $N\ge 3$. Then 
		\begin{enumerate}
			\item The first eigenvalue $\Lambda_1$ of $\Delta_{\gamma^0}$ is  
			\[
			\Lambda_1= k^2_1,\qquad k_1= \arccos \frac {N-1} {N}.
			\]
			It is  a simple eigenvalue and the corresponding eigenfunction is 
			\begin{equation} \label{psi0}
				\psi_0(t)=
				\begin{cases}
					\sin k_1t  &\text{on } (u_k,v_k),\ \forall 1\leq k\leq N,\\
					\sqrt{\frac{2}{N}}\cos k_1(t-\frac 12) &\text{on } (v_i,v_j),\ \forall 1\leq i< j\leq N,
				\end{cases}
			\end{equation}
			\item The second eigenvalue $\Lambda_2$ of $\Delta_{\gamma^0}$ is 
			\[
			\Lambda_2=k^2_2,\qquad k_2= \arccos \frac {-1} {N}.
			\]
			It has  multiplicity $N-1$, and the corresponding eigenspace is spanned by  functions $(\psi_1,\cdots,\psi_N)$ (which has dimension $N-1$), where 
			\[
			\psi_i(t)=
			\begin{cases}
				\sin k_2t \qquad &\text{on } (u_i,v_i), \\
				-\frac{1}{N-1} \sin k_2t \qquad &\text{on } (u_k,v_k),\ \forall k\neq i, \\
				\frac {\sqrt{N^2-1}} {N} \cos k_2t- \frac {1} {N(N-1)} \sin k_2t \qquad &\text{on } (v_i,v_j),\ \forall i<j, \\
				\frac {\sqrt{N^2-1}} {N} \cos k_2(1-t)- \frac {1} {N(N-1)} \sin k_2(1-t) \quad &\text{on } (v_j,v_i),\ \forall j<i,\\				
				-\frac {\sqrt{2N+2}} {(N-1)\sqrt N} \cos k_2(t-\frac 12) \qquad &\text{on other } (v_j,v_k). 			
			\end{cases}
			\]
		\end{enumerate}  
	\end{lemma}
	
	\begin{proof}
		The proof is similar to the argument in \cite[Theorem II.1]{CV2}. Firstly, it is easy to check the trigonometric identities
		\[\sin k_1=\sqrt{\frac 2 N}\cos\frac{k_1}{2}, \qquad \sin k_2=\frac{\sqrt{N^2-1}}{N}\]
		and 
		\[-\frac 1{N-1}\sin k_2=-\frac{\sqrt{2N+2}}{(N-1)\sqrt N}\cos\frac{k_2}2=\frac{\sqrt{N^2-1}}{N}\cos k_2-\frac 1{N(N-1)}\sin k_2,\]
		which ensure the continuity of the functions $\{\psi_0,\psi_1,\cdots,\psi_N\}$.  
		It is also easy to check 
		\[\cos k_1-(N-1)\sqrt{\frac{2}{N}}\sin \frac{k_1}{2}=0\]
		and
		\[
		-\frac{N+1}{N(N-1)} \cos k_2 - \frac{\sqrt{N^2-1}}{N} \sin k_2+ \frac{(N-2)\sqrt{2N+2}}{(N-1)\sqrt{N}}\sin\frac{k_2}{2}=0,
		\]
		which guarantee that the functions $\{\psi_0, \psi_1,\cdots, \psi_N\}$ belong to the domain of $\Delta_{\gamma^0}$. 
		
		It is also direct to check that $\{\psi_0, \psi_1,\cdots, \psi_N\}$ are eigenfunctions of $\Delta_{\gamma^0}$ with eigenvalues $\Lambda_0$ and $\Lambda_1$.  Although the functions $\{\psi_1,\cdots,\psi_{N}\}$ are constrained by the relation 
		\[
		\sum_{j=1}^N \psi_j =0,
		\]
		the functions $\{\psi_1,\cdots,\psi_{N-1}\}$ are linearly independent. So  
		\[
		\dim [\mathrm{span} (\psi_1,\cdots,\psi_N)]= N-1.
		\]
		
		It remains to prove  $\lambda_{N+1} (\Delta_{\gamma^0})$ (the $(N+1)$-th eigenvalue of $\Delta_{\gamma^0}$) is greater than $\Lambda_2$. For this purpose, we bisect all edges $\{(v_i,v_j)\}_{1\leq i< j\leq N}$ at their midpoints, decomposing $G_{N}^{\gamma^0}$ into $N$ identical new metric graphs (which are star graphs). Let $\Gamma$ be the new graph containing the edge $(u_1,v_1)$. Let $\Delta_\Gamma$ be the Laplacian defined on $\Gamma$ with Dirichlet boundary condition at $u_1$ and with Neumann boundary condition at all other boundary points. Let $\mu_2$ be the second eigenvalue of the Laplacian $\Delta_\Gamma$. By the min-max principle,  
		\[
		\lambda_{N+1} (\Delta_{\gamma^0})\geq \mu_2.
		\]
		
		Finally we prove $\mu_2=\pi^2$. We parameterize the edge $(u_1,v_1)$ as $[0,1]$ and each bisected edge $(w_i,v_1)$ as $[0,1/2]$, where $w_i$ is the midpoint of $(v_1,v_i)$. Let $k^2$ ($k>0$) be an eigenvalue of $\Delta_\Gamma$, with corresponding eigenfunction $f$ of the form
		\[
		f(t)=
		\begin{cases}
			\sin kt  &\text{on } (u_1,v_1), \\
			C\cos kt &\text{on } (w_i,v_1),\ \forall i\neq 1.
		\end{cases}
		\]
		The constants $k$ and $C$ must satisfy 
		\begin{equation} \label{equ1}
			\begin{cases}
				\sin k= C \cos \frac{k}{2}, \\
				\cos k-(N-1) C \sin \frac{k}{2}=0.
			\end{cases}
		\end{equation}
		There are two classes of solutions.
		If $\cos \frac{k}{2} = 0$, then $k = (2m+1)\pi$ with $m \in \mathbb{Z}_{\geq 0}$, leading to eigenvalues $(2m+1)^2\pi^2$. If $\cos \frac{k}{2} \neq 0$, then $C = 2\sin\frac{k}{2}$. Substituting this into the second equation in \eqref{equ1} yields $N\cos k = N-1$, which gives eigenvalues $k^2$ with 
		\[
		k = \pm \arccos \frac{N-1}{N} + 2m\pi, \quad m \in \mathbb{Z}_{\geq 0}.
		\] 
		Since there is only one value $k^2<\pi$, we conclude $\mu_2 = \pi^2$ and thus
		$\lambda_{N+1} (\Delta_{\gamma^0}) > \Lambda_2$.
	\end{proof}

	Now consider the map 
	$I_\gamma: L^2(G_N^{\gamma^0})\to L^2(G_N^{\gamma})$ which maps a function  $f$ to the function $I_\gamma(f)$ defined by 
	\begin{equation}\label{isometry}
		\begin{aligned}
			& I_\gamma(f)_{ij}\big(\int_0^t \gamma_{ij}(s)\mathrm{d}s\big)= f_{ij}(t)\cdot \gamma_{ij}(t)^{-\frac 12}, \qquad  \forall t\in [0,1]\\
			&I_\gamma(f)_k \big(\int_0^t \gamma_k(s)\mathrm{d}s\big)= f_k(t)\cdot \gamma_k(t)^{-\frac 12}, \qquad \forall t\in [0,1].
		\end{aligned}
	\end{equation}
	It is easy to check that $I_\gamma$ is an isometry. As a result, the spectrum of $\Delta_\gamma$ is the same as the spectrum of the Friedrichs extension of the quadratic  form
	\begin{equation} \label{q_gamma}
		\bar q_\gamma (f)= \sum_{k=1}^N \!\int_0^1 \! |(f_k\gamma_k^{-\frac{1}{2}})'(t)|^2 \gamma_k(t)^{-1} \mathrm{d}t +\! \sum_{  1\le i<j\le N }
		\! \int_0^1 \! |(f_{ij}\gamma_{ij}^{-\frac{1}{2}})'(t)|^2 \gamma_{ij}(t)^{-1} \mathrm{d}t
	\end{equation}
	on $H^1_0(G_N^{\gamma^0})$. The advantage of $\bar q_\gamma$ is that its domain is independent of $\gamma$. We denote the operator associated with $\bar q_\gamma$ by $\bar \Delta_\gamma$. Note that by construction, $\bar q_{\gamma^0}$ is $q_{\gamma^0}$.

    Choose $\frac{N(N+1)}{2}$ positive bump functions in $C^\infty_0(0,1)$,
	\[
	\begin{aligned}
		\varphi_{ij} \quad (1\leq i<j\leq N)\qquad \text{and}  \qquad  \varphi_k \quad (\ 1\leq k\leq N),
	\end{aligned}
	\]
	satisfying
	\[
	\int_0^1 \varphi_{ij}(t) \mathrm{d}t= \int_0^1 \varphi_k(t) \mathrm{d}t= 1.
	\]
	For any perturbation parameters $x=(x_{ij}, x_k)_{i<j, 1\leq i,j,k \leq N}$ near the origin of $\mathbb{R}^{\frac{N(N+1)}{2}}$, let $\gamma^x$ be the perturbed metric element on $G_N$ defined by
	\[
	\gamma^x_{ij}= 1+x_{ij} \varphi_{ij},\qquad \gamma^x_k= 1+x_k \varphi_k.
	\]
	Under such small metric perturbation, the operator $\bar \Delta_{\gamma^x}$ has $N-1$ eigenvalues (counting multiplicity) near  $\Lambda_2$. Let $E_x$ be the $(N-1)$-dimensional subspace spanned by the eigenfunctions associated with these eigenvalues. Note that $E_0$ is the eigenspace of $\Lambda_2$. Let 
	\[U_x: E_0 \to E_x \]
	be  the isometry constructed in  \eqref{iso}, then we get a quadratic form on $E_0$, 
	\[
	Q_x= \bar q_{\gamma^x} \circ U_x,
	\]
	where $\bar q_{\gamma^x}$ defined in \eqref{q_gamma}. We denote by  $Q(E_0)$  the space of all quadratic forms on $E_0$. 
	
	\begin{proposition} \label{submer}
    	For any sufficiently small ball $B \subset \mathbb{R}^{\frac{N(N+1)}{2}}$ centered at the origin, the map
		\[
		F: B \to Q(E_0), \qquad x \mapsto Q_x
		\]
		is a submersion at the origin.
	\end{proposition}
	
	\begin{proof}
		Again, we follows the strategy of the proof of \cite[Theorem II.2]{CV2}. For any perturbation parameter $x$, we denote
		\[
		\dot Q_x= \frac{d}{ds}\bigg|_{s=0} Q_{sx}, \qquad \dot {\bar q}_x= \frac{d}{ds}\bigg|_{s=0} \bar q_{\gamma^{sx}},
		\]
		As in the proof of \cite[Lemma II.3]{CV2}, for any $\varphi,\psi\in E_0$, one has
		\begin{equation}\label{Qq0}
		\begin{aligned}
		 \dot Q_x(\varphi,\psi)=& \frac{d}{ds}\bigg|_{s=0}  \bar q_{\gamma^{sx}}(U_{sx}\varphi,U_{sx}\psi)\\
		= &\dot {\bar q}_x (\varphi,\psi)+ q_{\gamma^0}(\frac{d}{ds}\bigg|_{s=0}U_{sx}\varphi, \psi)+q_{\gamma^0}(\varphi,\frac{d}{ds}\bigg|_{s=0}U_{sx}\psi)\\
		= &\dot {\bar q}_x (\varphi,\psi).
	    \end{aligned}
		\end{equation}
		
		Now fix any $f \in E_0$. On each edge of the graph $G_N^{\gamma^0}$, $f$ is a linear combination of $\cos k_2t$ and $\sin k_2t$. So for any $1\le k\le N$,  
		\[
		\begin{aligned}
			&\frac{d}{ds}\bigg|_{s=0}\int_0^1 |(f_k(t)(1+sx_k\varphi_k(t))^{-\frac 12})'|^2 (1+sx_k\varphi_k(t))^{-1} \mathrm{d}t\\
			=&\int_0^1 2f_k'(t)\frac{d}{ds}\bigg|_{s=0}[((1+sx_k\varphi_k(t))^{-\frac 12})']\mathrm{d}t-\int_0^1 f_k'(t)^2 x_k \varphi_k(t) \mathrm{d}t\\
			=&\int_0^1 2f_k'(t) (-\frac 12 f_k(t)x_k\varphi_k(t))'\mathrm{d}t-\int_0^1 f_k'(t)^2 x_k \varphi_k(t) \mathrm{d}t\\
			=&\int_0^1 f_k''(t) f_k(t)x_k\varphi_k(t)-f_k'(t)^2 x_k \varphi_k(t) \mathrm{d}t\\
			=&-\int_0^1 (f_k'(t)^2+k^2_2 f_k(t)^2)x_k \varphi_k(t) \mathrm{d}t,
		\end{aligned}
		\] 
		and similarly, for any $1\le i<j\le N$,
		\[
		\begin{aligned}
			&\frac{d}{ds}\bigg|_{s=0}\int_0^1 |(f_{ij}(t)(1+sx_{ij}\varphi_{ij}(t))^{-\frac 12})'|^2 (1+sx_{ij}\varphi_{ij}(t))^{-1} \mathrm{d}t\\
			=&-\int_0^1(f'_{ij}(t)^2+ k^2_2 f_{ij}(t)^2) x_{ij} \varphi_{ij}(t) \mathrm{d}t.
		\end{aligned}
		\]
		Thus, by expression \eqref{q_gamma} and \eqref{Qq0}, one has	
		\begin{equation} \label{qx}
			\begin{aligned}
				\dot Q_x(f)=\dot {\bar q}_x (f)=& -\sum_{k=1}^N \int^1_0 (f_k'(t)^2+k^2_2 f_k(t)^2)x_k \varphi_k(t) \mathrm{d}t\\
				&- \sum_{1\leq i< j\leq N} \int_0^1 (f'_{ij}(t)^2+ k^2_2 f_{ij}(t)^2) x_{ij} \varphi_{ij}(t) \mathrm{d}t.
			\end{aligned}
		\end{equation}
		
		Let $\{X_k=(0, e_k), X_{ij}=(e_{ij},0)\}_{k,i<j}$  be a canonical basis of $\mathbb{R}^{\frac{N(N+1)}{2}}$, where $e_k$ is the vector in $\mathbb{R}^N$ which equals 1 only on the $k$-th component, and $e_{ij}$ is the vector in $\mathbb{R}^{\frac{N(N-1)}{2}}$ which equal 1 only on the $(i,j)$-th component. Considering the vectors
		\[
		F_k=\big( \dot {\bar q}_{X_k}(\psi_{l}, \psi_{m}) \big)_{1\leq l< m\leq N}, \qquad F_{ij}=\big( \dot {\bar q}_{X_{ij}}(\psi_{l}, \psi_{m}) \big)_{1\leq l< m\leq N}
		\]
		where $\psi_{l},\psi_{m}$ are the eigenfunctions defined in Lemma \ref{spec of gamma0}. To prove Proposition \ref{submer}, it suffices to verify that
		\begin{equation} \label{rank}
			\rank\{ F_k, F_{ij} \}_{k,i<j}=\frac{N(N-1)}{2}.
		\end{equation}
		By expression of $\psi_i$ in Lemma \ref{spec of gamma0} and (\ref{qx}), one has
		\begin{itemize}
			\item $\dot {\bar q}_{X_k}(\psi_l,\psi_m)= -\frac{k_2^2}{(N-1)^2}$,\qquad if $l,m\neq k$
			\item $\dot {\bar q}_{X_k}(\psi_l,\psi_k)=\dot {\bar q}_{X_k}(\psi_k,\psi_l)=\frac{k_2^2}{N-1}$,\qquad if $l\neq k$
			\item $\dot {\bar q}_{X_{ij}}(\psi_i,\psi_j)= k_2^2\frac{N^2-2}{N(N-1)^2}$
			\item $\dot {\bar q}_{X_{ij}}(\psi_l,\psi_m)= -k_2^2 \frac{2(N+1)}{N(N-1)^2}$,\qquad if $\{l,m\}\cap \{i,j\}=\emptyset$
			\item $\dot {\bar q}_{X_{ij}}(\psi_i,\psi_m)=\dot {\bar q}_{X_{ij}}(\psi_l,\psi_j)=k_2^2 \frac{N^2-N-2}{N(N-1)^2}$,\qquad if $l\neq i, m \neq j$.
		\end{itemize}
		Let $e$ be the vector in $\mathbb{R}^{\frac{N(N-1)}{2}}$ which equals 1 on each component. Then (\ref{rank}) is a consequence of  the following facts:
		\begin{itemize}
			\item $\sum\limits_{k=1}^N F_k= \frac{k^2_2 N}{(N-1)^2} \cdot e$,
			\item $F_{ij}+ k_2^2\frac{2(N+1)}{N(N-1)^2}\cdot e-\frac{N+1} N(F_i+F_j+\frac{2k^2_2}{(N-1)^2}\cdot e)=-k^2_2\frac N{(N-1)^2}\cdot e_{ij}$.
		\end{itemize}
	\end{proof}
	
	Next, we present a proposition concerning the stability of $\lambda_1(\bar \Delta_{\gamma^x})$, which will be utilized in the proof of Theorem \ref{prediffeigen}.
	
	\begin{proposition}\label{simple stable}
		Let $B \subset \mathbb{R}^{\frac{N(N+1)}{2}}$ be a ball centered at the origin with sufficiently small radius. Then the map
		\[
		F_1: B\to \mathbb R_{>0},\qquad x\mapsto \lambda_1(\bar \Delta_{\gamma^x})
		\]
		is a submersion at the origin.
	\end{proposition}
	
	\begin{proof}
		For a point $x$ and $s>0$, let $\psi^{sx}_0$ be the normalized first eigenfunction of $\bar\Delta_{\gamma^{sx}}$. Since the quadratic form $\bar q_{\gamma^{sx}}$ is differentiable with respect to $s$, one can multiply $\psi^{sx}_0$ by $\pm 1$ to ensure that it is differentiable with respect to $s$ and the following holds:
		\[
		\lim_{s\to 0^+}\psi_0^{sx}=\frac{\psi_0}{|\psi_0|_{L^2(G_N^{\gamma^0})}},
		\]
		where $\psi_0$ is defined in \eqref{psi0}. Define
		\[
		\psi_0^0=\frac{\psi_0}{|\psi_0|_{L^2(G_N^{\gamma^0})}},\qquad \alpha_x(s)=\langle \psi^{sx}_0,\psi^0_0\rangle_{L^2(G_N^{\gamma^0})}.
		\]
		Hereafter, we omit the subscript $L^2(G_N^{\gamma^0})$ for simplicity. Since $|\psi^{sx}_0|=1$, we have
		\[
		\begin{aligned}
		0&=\frac d{ds}\bigg|_{s=0} \langle\psi^{sx}_0,\psi^{sx}_0 \rangle\\&=\frac d{ds}\bigg|_{s=0} \big(\alpha_x(s)^2+\langle \psi^{sx}_0-\alpha_x(s)\psi^0_0, \psi^{sx}_0-\alpha_x(s)\psi^0_0\rangle\big)\\
		&=2\frac d{ds}\bigg|_{s=0} \alpha_x(s).
		\end{aligned}
		\]
		Thus,
		\[
		\begin{aligned}
		\frac d{ds}\bigg|_{s=0} \lambda_1(\bar \Delta_{\gamma^{sx}})=& \frac d{ds}\bigg|_{s=0} \bar q_{\gamma^{sx}}(\psi^{sx}_0)\\=&\frac d{ds}\bigg|_{s=0}\big( \bar q_{\gamma^{sx}}(\psi^0_0)+2q_{\gamma^0}(\psi^{sx}_0,\psi^0_0) \big)\\
		=&\frac d{ds}\bigg|_{s=0} \frac{\bar q_{\gamma^{sx}}(\psi_0)}{|\psi_0|^2_{L^2(G_N^{\gamma^0})}}+2 \frac d{ds}\bigg|_{s=0} [\alpha_x(s) q_{\gamma^0}(\psi_0^0)]\\
		=&\frac d{ds}\bigg|_{s=0} \frac{\bar q_{\gamma^{sx}}(\psi_0)}{|\psi_0|^2_{L^2(G_N^{\gamma^0})}}.
		\end{aligned}
		\] 
		By \eqref{psi0}, if we take $x=(0,e_k)$ for some $1\le k\le N$, then 
		\[
		\frac d{ds}\bigg|_{s=0}\bar q_{\gamma^{sx}}(\psi_0)=-(\arccos\frac{N-1}{N})^2\neq 0.
		\]
		Thus $F_1$ is a submersion at the origin.
	\end{proof}
	
	\section{Proof of Theorem \ref{prevalueRobin} and Theorem \ref{2dimRobin}} \label{proof1}
	
	In this section, we will employ concepts from spectral convergence theory and stable metrics. For precise definitions, we refer the reader to Appendix \ref{app}.
	
	We will reformulate the problem of prescribing Robin eigenvalues as an equivalent problem of prescribing Dirichlet eigenvalues. The first key step involves reducing Theorem \ref{prevalueRobin} to the case where the Robin parameter function $\rho$     has a sufficiently large lower bound. This reduction is established in the following proposition, which  is inspired by \cite[Theorem 3.6]{PJSte}.

	\begin{proposition} \label{DtoR}
		Let $(M,g)$ be a compact Riemannian manifold of dimension $n \geq 2$, with smooth boundary $\partial M$, and let $\bar \rho, \rho \in C_+(\partial M)$ be two functions satisfying $\bar \rho > \rho$. Then there exists a family of  Riemannian metrics $\{g_\varepsilon\}_{\varepsilon>0}$  in the conformal class of $g$ on $M$ such that
		\begin{itemize}
			\item[(1)] $\lim_{\varepsilon \to 0^+} \lambda_k(M, g_\varepsilon,\rho) =\lambda_k(M, g,\bar \rho)$ for all $k$.
			\item[(2)] $g_\varepsilon= (\bar \rho/\rho)^{2/(n-1)} g$ on $\partial M$.
			\item[(3)] $g_\varepsilon \to g$  (in $C^0$ topology)  uniformly on any compact subset in $M \setminus \partial M$.
		\end{itemize}
		Moreover, if the first $N+1$ eigenvalues $\{\lambda_k(M, g,\bar\rho)\ |\ 1\le k \le N+1\}$ satisfies the hypothesis (\ref{hoe}) in appendix A, then the $N$-spectral difference between the Robin Laplacian $\Delta(M,g_\varepsilon, \rho)$ and the Robin Laplacian $\Delta(M,g, \bar \rho)$ converges to 0 as $\varepsilon \to 0^+$.
	\end{proposition}
	
	\begin{proof}
		Following the strategy of the proof of \cite[Theorem 3.6]{PJSte}, we first choose a family of smooth functions $h_\varepsilon\in C^\infty(M\setminus \partial M)$ such that
		\begin{itemize}
			\item[(1)] $h_\varepsilon= (\bar \rho/\rho)^{1/(n-1)}$ on $\partial M$, $h_\varepsilon \geq 1$ in $M$,
			\item[(2)] As $\varepsilon \to 0^+$, $h_\varepsilon\to  1$ uniformly on any compact subset in $M \setminus \partial M$.
		\end{itemize}
		For example, one may start with the harmonic extension $\hat\rho$ of $(\bar \rho/\rho)^{1/(n-1)}$ on $M$, then let $h_\varepsilon=\hat \rho$ in the $\varepsilon$-neighborhood of $\partial M$ and  $h_\varepsilon=1$ away from $2\varepsilon$-neighborhood of $\partial M$.
		Consider the conformal metrics
		\[
		g_\varepsilon = h_\varepsilon^2 \cdot g.
		\]
		Then the quadratic form associated to $\Delta(M, g_\varepsilon,\rho)$ is 
		\[
		Q_\varepsilon (f)= \int_M |\nabla_g f|^2 h_\varepsilon^{n-2} \mathrm{d} V_g+ \int_{\partial M} \bar \rho \cdot f^2 \mathrm{d}\sigma_g,
		\]
		while the $L^2$-norm associated to $g_\varepsilon$ is 
		\[
		|f|_\varepsilon= (\int_M f^2 h_\varepsilon^n \mathrm{d}V_g)^{\frac{1}{2}}.
		\]
		Let
		\[
		Q(f)=\int_M |\nabla_g f|^2 \mathrm{d} V_g+\int_{\partial M} \bar\rho\cdot f^2\mathrm{d}\sigma_g,\qquad |f|=(\int_M f^2 \mathrm{d}V_g)^{\frac{1}{2}}.
		\]
		Then 
		there exists   $C_1>0,C_2>0$ such that
		\[
		C_1 |f|\leq |f|_\varepsilon \leq C_2|f|,\qquad \forall f\in L^2(M).
		\]
		Applying Theorem \ref{ct2} to $Q_\varepsilon,Q,|\cdot|_\varepsilon$  and $|\cdot|$, one gets Theorem \ref{DtoR}. 
	\end{proof}
	
	Proposition \ref{DtoR} establishes that finite eigenvalues of the Robin Laplacian on $M$ with Robin function $\rho$ can be prescribed, provided there exists a certain stable family that can be used to prescribe finite eigenvalues of the Robin Laplacian on $M$ with any boundary function $\bar\rho$ satisfying  $\bar\rho > \rho$.
	
	In the proof of \cite[Theorem 1.2]{HW}, we constructed a family of stable metrics  
	\begin{equation} \label{mathcalF}
		\mathcal F: B \to \mathcal M(M)
	\end{equation}  
	for the Dirichlet Laplacian on $M$ with prescribed volume $V$ and prescribed first $N$ eigenvalues. A natural question arises: is $\mathcal F(B)$ also a stable family for the Robin Laplacian on $M$? To investigate this, we will compare the eigenspaces of the Robin Laplacian $\Delta(M, g,\rho)$ and of the Dirichlet Laplacian $\Delta^\mathcal{D}(M,g)$, in the case the boundary function $\rho$ is sufficiently  large.
	For simplicity, we denote the eigenvalues of $\Delta^\mathcal{D}(M,g)$ by
	\begin{equation} \label{Dirieigen}
		\lambda_k^\mathcal{D} (M,g), \qquad k\ge 1.
	\end{equation}

	\begin{proposition} \label{Diri-Robin}
		Let $(M,g)$ be a compact Riemannian manifold with smooth boundary $\partial M$, and $\rho \in C_+(\partial M)$ be a positive Robin parameter. Suppose the first $N+1$ Dirichlet eigenvalues $\{\lambda^\mathcal{D}_k (M,g)\ | \ 1 \le k \le N+1\}$ satisfy the hypothesis (\ref{hoe}). Then for any $\alpha > 0$, there exists $C(\alpha)>0$ such that whenever $\rho \ge C(\alpha)$ on $\partial M$, the $N$-spectral difference between $\Delta (M, g, \rho)$ and $\Delta^\mathcal{D} (M,g)$ is less than $\alpha$.
	\end{proposition}
	
	\begin{proof}
		Suppose $\rho \ge C$ on $\partial M$. By Proposition 4.5 in \cite{BFK},  for any $k$,
		\begin{equation}\label{conRobDiri}
		\lambda^\mathcal{D}_k (M,g) - \lambda_k(M, g,\rho)\to 0,\qquad\text{as }C \to +\infty.
		\end{equation}
		So if $C$ is sufficiently large, the $N$-eigenspace of $\Delta(M, g,\rho)$ is well-defined (for the definition of the $N$-eigenspace, see Appendix \ref{app}). In what follows we denote the $N$-eigenspace of $\Delta(M, g,\rho)$ by $E_0$, and denote the $N$-eigenspace of $\Delta^\mathcal{D} (M,g)$   by $E_1$.
		
		According to Lemma \ref{crit} (applied to $A_0=A_1=I$), \textbf{it suffices to prove that $\|B\| \to 0$ as $C \to +\infty$} (see Appendix A for the definition of the operator $B$). For simplicity, we assume
		\[
		a=\lambda^\mathcal D_1(M,g)<\lambda^\mathcal D_2(M,g)=\cdots=\lambda^\mathcal D_N(M,g)=b<b+\delta\le \lambda^\mathcal D_{N+1}(M,g)
		\]
		where $\delta>0$ is the constant in \eqref{hoe}.  The general case (where eigenvalues may have different multiplicities or clustering patterns) can be handled by applying the same argument iteratively to each relevant spectral cluster. We omit these routine technical details for brevity.
		
		Let $\{\varphi^\mathcal D_k\}_{k\ge1}$ be an orthonormal set of eigenfunctions corresponding to $\{\lambda^\mathcal D_k(M,g)\}_{k\ge 1}$, and let $\{\varphi^\rho_k\}_{k\ge 1}$ be an orthonormal set of eigenfunctions corresponding to $\{\lambda_k(M, g,\rho)\}_{k\ge 1}$. To prove that $\|B\| \to 0$ as $C \to +\infty$, \textbf{it suffices to show that $P^\rho_N \varphi^\mathcal D_k \to\varphi^\mathcal D_k$  for each $1\le k\le N$}, where $P^\rho_N$ is the  orthogonal projection onto  the subspace
		\[\Span\{\varphi^\rho_1,\cdots,\varphi^\rho_N\}.\]
		
		Let $Q_\mathcal D$ and $Q_\rho$ denote the quadratic forms associated with $\Delta^\mathcal{D} (M,g)$ and $\Delta (M, g, \rho)$, respectively. Then,
		\[
		\begin{aligned}
		a=&Q_\mathcal D(\varphi^\mathcal D_1)=Q_\rho (\varphi^\mathcal D_1)=Q_\rho(\langle \varphi^\mathcal D_1,\varphi^\rho_1 \rangle_{L^2}\cdot \varphi^\rho_1)+Q_\rho(\varphi^\mathcal D_1-\langle \varphi^\mathcal D_1,\varphi^\rho_1 \rangle_{L^2}\cdot \varphi^\rho_1)\\
		\ge&\langle \varphi^\mathcal D_1,\varphi^\rho_1 \rangle_{L^2}^2\cdot \lambda_1(M,g,\rho)+\lambda_2(M,g,\rho)(1-\langle \varphi^\mathcal D_1,\varphi^\rho_1 \rangle_{L^2}^2),
		\end{aligned}
		\]
		which implies 
		\[
		1\ge \langle \varphi^\mathcal D_1,\varphi^\rho_1 \rangle_{L^2}^2\ge \frac{\lambda_2(M,g,\rho)-a}{\lambda_2(M,g,\rho)-\lambda_1(M,g,\rho)}.
		\]
		Combining this with \eqref{conRobDiri}, one has 
		\begin{equation}\label{lambda1}
		\langle \varphi^\mathcal D_1,\varphi^\rho_1 \rangle_{L^2}^2\to 1\qquad \text{as }C\to +\infty,
		\end{equation}
		which in turn implies
		\begin{equation}\label{phiD1}
		|\varphi_1^\mathcal D-P^\rho_N \varphi_1^\mathcal D|^2_{L^2}=|\varphi_1^\mathcal D|^2_{L^2}-|P^\rho_N \varphi_1^\mathcal D|^2_{L^2}\le|\varphi_1^\mathcal D|_{L^2}^2-\langle \varphi^\mathcal D_1,\varphi^\rho_1 \rangle_{L^2}^2\to0,\qquad \text{as }C\to \infty.
		\end{equation}
		
		Next, define the subspace:
		\[
		E^*=\Span\{\varphi^\rho_2,\cdots,\varphi^\rho_N\}.
		\] 
		Let
		\[
		\widetilde{\varphi^\mathcal D_k},\qquad 2\le k\le N
		\]
		denote the projection of $\varphi^\mathcal D_k$ onto $E^*$. Then, for all $2\le k\le N$,
		\begin{equation*} 
		\begin{aligned}
		b=Q_\mathcal D(\varphi^\mathcal{D}_k)=&Q_\rho(\varphi^\mathcal{D}_k)\\
		=&Q_\rho(\langle\varphi^\mathcal{D}_k,\varphi^\rho_1\rangle_{L^2}\cdot \varphi^\rho_1)+Q_\rho(\widetilde{\varphi^\mathcal D_k})+Q_\rho(\varphi^\mathcal{D}_k-\langle\varphi^\mathcal{D}_k,\varphi^\rho_1\rangle_{L^2}\cdot \varphi^\rho_1-\widetilde{\varphi^\mathcal D_k})\\
		\ge&\lambda_1(M,g,\rho)\cdot \langle\varphi^\mathcal{D}_k,\varphi^\rho_1\rangle_{L^2}^2+\lambda_2(M,g,\rho)\cdot |\widetilde{\varphi^\mathcal D_k}|_{L^2}^2+\\
		&\lambda_{N+1}(M,g,\rho)(1-\langle\varphi^\mathcal{D}_k,\varphi^\rho_1\rangle_{L^2}^2-|\widetilde{\varphi^\mathcal D_k}|_{L^2}^2),
		\end{aligned}
		\end{equation*}
		which implies
		\begin{equation}\label{2N}
			1\ge |\widetilde{\varphi^\mathcal D_k}|_{L^2}^2\ge \frac{\lambda_1(M,g,\rho)\cdot \langle\varphi^\mathcal{D}_k,\varphi^\rho_1\rangle_{L^2}^2+\lambda_{N+1}(M,g,\rho)(1-\langle\varphi^\mathcal{D}_k,\varphi^\rho_1\rangle_{L^2}^2)-b}{\lambda_{N+1}(M,g,\rho)-\lambda_2(M,g,\rho)}.
		\end{equation}
		From \eqref{lambda1}, one has that for all $2\le k\le N$,
		\[
		\langle\varphi^\mathcal{D}_k,\varphi^\rho_1\rangle_{L^2}^2\to 0,\qquad \text{as }C\to +\infty.
		\]
		Combining this with \eqref{2N} and \eqref{conRobDiri}, we conclude that for all $2\le k\le N$,
		\begin{equation}\label{con2N}
			|\widetilde{\varphi^\mathcal D_k}|_{L^2}^2\to 1,\qquad \text{as }C\to +\infty.
	    \end{equation} 
	    Thus for all $2\le k\le N$,
	    \begin{equation}\label{2k}
	    	\begin{aligned}
	    		|\varphi_k^\mathcal D-P^\rho_N\varphi_k^\mathcal D|^2= |\varphi_k^\mathcal D|^2-|P^\rho_N\varphi_k^\mathcal D|^2\le |\varphi_k^\mathcal D|^2-|\widetilde{\varphi^\mathcal D_k}|^2\to0,\qquad\text{as }C\to\infty.
	    	\end{aligned}
	    \end{equation}
	    
	    Finally, combining \eqref{phiD1} and \eqref{2k}, the proof is complete.
	\end{proof}
	
	Thus by Proposition \ref{Diri-Robin}, if we assume $\rho$ is a sufficiently large constant function, then the family  $\mathcal F(B)$ defined in \eqref{mathcalF} remains stable  for the Robin Laplacian on $M$ with boundary parameter   $\rho$. Combining this with Proposition \ref{DtoR}, one can prescribe finite eigenvalues of the Robin Laplacian on $M$ with arbitrary positive Robin parameter $\rho \in C_+(\partial M)$:
	
	\begin{proof}[Proof of Theorem \ref{prevalueRobin} and Theorem \ref{2dimRobin}]
		Let $C$ be a sufficiently large constant, and consider the constant boundary function  $\rho_C(x)=C$ on $\partial M$. According Proposition \ref{DtoR} (with $\bar\rho=\rho_C$), one immediately get a stable family of the Robin Laplacian on $M$ with boundary function $\rho$:
		\[
		F^\varepsilon: B\to \mathcal M(M),\qquad x\to h_\varepsilon^2\cdot \mathcal F(x),
		\]
		where $h_\varepsilon$ is the conformal factor in the proof of Proposition \ref{DtoR}, and $\mathcal F(x)$ is the family  of stable metrics given in \eqref{mathcalF}. By construction, $h_\varepsilon \to 1$ uniformly on any compact subset in $M \setminus \partial M$. So we can multiply $F^\varepsilon(x)$ by a factor very close to 1, to obtain metrics with prescribed volume $V$ while maintaining stability for  sufficiently small $\varepsilon>0$. Thus, we can prescribe both the first $N$ eigenvalues and the volume of $M$ simultaneously, completing the proof of Theorem \ref{prevalueRobin}.
		
		The proof of Theorem \ref{2dimRobin} is very similar to that of Theorem \ref{prevalueRobin}: one only needs to replace the family of stable metrics in the proof of \cite[Theorem 1.2]{HW} by the family of stable metrics in the proof of \cite[Theorem 1.1]{H}. 
	\end{proof}

	\section{Proof of Theorem \ref{premulRobin} and Theorem \ref{prediffeigen}}\label{proof2}
	
	We proceed by the same strategy. Namely, we first examine the corresponding Dirichlet eigenvalue problem and then apply Proposition \ref{DtoR} and Proposition \ref{Diri-Robin} to establish the desired results.
		
	To achieve this, for each point $x$, we associate to the graph $G_N^{\gamma^x}$ (defined in \S\ref{Qgraph}) a family of Riemannian manifolds $\Omega_\varepsilon^x$, constructed by gluing Euclidean open sets along specified boundaries. We refer to such manifolds as locally Euclidean manifolds. Then we impose a special mixed boundary condition on $\Omega_\varepsilon^x$ and consider corresponding Laplacian. The key is to prove that the Laplacian on  $\Omega_\varepsilon^x$ with this mixed boundary condition and the operator $\bar \Delta_{\gamma^x}$ (defined in \S\ref{Qgraph}) on $G_N^{\gamma^x}$ has  $N$-spectral difference converging to 0 as $\varepsilon \to 0$.

	To construct $\Omega_\varepsilon^x$, we start with $S_{N+1}$, the star graph with $N+1$ vertices. Let  $W_{N+1}$ be the radius 1 Euclidean tubular neighborhood  of  $S_{N+1}$, constructed in \S IV of \cite{CV2} as follows: Choose a constant $K$ such that the sphere  $S^{n-1}_K \subset \mathbb R^{n}$ (of radius $K$) admits  at least $N$ points with pairwise distances strictly greater than 2. Embed $S_{N+1}$ into $\mathbb R^n$ by placing the center vertex at the origin of $\mathbb R^n$, and placing the remaining vertices on $S^{n-1}_K$ such that their pairwise distances are greater than 2. Then, $W_{N+1}$ can be taken as the Euclidean tubular neighborhood of radius 1 of $S_{N+1}$ in $\mathbb R^n$. 
	
	For any $1\le l\le N$, let  $X_{l,\varepsilon}$ be a scaled copy of $W_{N+1}$ by a factor $\varepsilon$. Given any point $x=(x_{ij},x_k)$ in $\mathbb{R}^{\frac{N(N+1)}{2}}$, we construct the locally Euclidean manifold  $\Omega_\varepsilon^x$ as follows: First for any  $1\leq i<j \le N$ we 	
	connect $X_{i,\varepsilon}$ and $X_{j,\varepsilon}$ by the Euclidean cylinder
	\[
	C_{ij,\varepsilon}= [0,1+x_{ij}] \times B^{n-1}(\varepsilon),
	\]
	where $B^{n-1}(\varepsilon)$ is the Euclidean ball in $\mathbb R^{n-1}$ of radius $\varepsilon$. Now for each $1\leq k\leq N$, the domain $X_{k,\varepsilon}$ has only one boundary component that is not connected to any cylinder $C_{ij,\varepsilon}$. We glue the Euclidean cylinder
	\[
	C_{k,\varepsilon}=[0, 1+x_k] \times B^{n-1}(\varepsilon)
	\]
	to this boundary curve along the end $\{1+x_k\}\times B^{n-1}(\varepsilon)$. The resulting locally Euclidean manifold  obtained by this   procedure is denoted by  $\Omega_\varepsilon^x$.
	
	\begin{proposition} \label{graph-domain}
		For any $\alpha >0$, there exists $\varepsilon_0=\varepsilon_0(x,\alpha)> 0$ such that for all $\varepsilon \in (0,  \varepsilon_0)$, the $N$-spectral difference between $\bar \Delta_{\gamma^x}$ and the Laplacian $\Delta_\Omega^{x,\varepsilon}$ of $\Omega^x_\varepsilon$ with ``Dirichlet boundary conditions on
		\[
		\{0\} \times B^{n-1}(\varepsilon)\text{ of $C_{k,\varepsilon}$}, \qquad \forall\ 1 \le k \le N
		\]
		and Neumann boundary conditions on all other boundary components" is less than $\alpha$.
	\end{proposition}
	
	\begin{proof}
		Under the isometry $I_\gamma$ in \eqref{isometry}, the operator $\bar \Delta_{\gamma^x}$ is isospectral to the operator  $\Delta_{\gamma^x}$. Therefore it suffices to prove that the $N$-spectral difference between $\Delta_{\gamma^x}$ and $\Delta_\Omega^{x,\varepsilon}$ tends to zero as $\varepsilon \to 0^+$. The proof follows a strategy similar to  Theorem IV.1 in \cite{CV2} (with necessary modifications).
		
		Consider the injective map 
		\begin{equation}
			J^x_\varepsilon:\ H^1_0(G^{\gamma^x}_N)\to H^1(\Omega^x_\varepsilon),\qquad f\mapsto J^x_\varepsilon(f)
		\end{equation}
		defined as follows:
		\begin{itemize}
			\item On $C_{ij,\varepsilon}$, for $y_1\in [0,1+x_{ij}]$ and $y'\in B^{n-1}(\varepsilon)$,
			\[
			J^x_\varepsilon(f)(y_1,y')=\big(\frac{1}{\omega_{n-1}\varepsilon^{n-1}}\big)^{\frac 12} f_{ij}(y_1)
			\]
			where $\omega_{n-1}$ denotes the volume of unit ball in $\mathbb R^{n-1}$.
			\item On $C_{k,\varepsilon}$, for $y_1\in [0,1+x_{k}]$ and $y'\in B^{n-1}(\varepsilon)$,
			\[
			J^x_\varepsilon(f)(y_1,y')=\big(\frac{1}{\omega_{n-1}\varepsilon^{n-1}}\big)^{\frac 12} f_k(y_1).
			\]
			\item For $y\in X_{l,\varepsilon}$,
			\[
			J^x_\varepsilon(f)(y)=\big(\frac{1}{\omega_{n-1}\varepsilon^{n-1}}\big)^{\frac 12} f(v_i).
			\]
		\end{itemize}
		It is straightforward to verify  
		\[
		\begin{cases}
			\int_{\Omega^x_\varepsilon} |J^x_\varepsilon(f)|^2 \mathrm d V=|f|_{L^2(G^{\gamma^x}_N)}^2+O(\varepsilon),\\
			q_{\Omega^x_\varepsilon}(J^x_\varepsilon(f)):=\int_{\Omega^x_\varepsilon} |\nabla J^x_\varepsilon(f)|^2 \mathrm d V=q_{\gamma^x}(f),
		\end{cases}
		\] 
		where $q_{\gamma^x}$ is the quadratic form defined in \S\ref{Qgraph}. Thus the $N$-spectral difference between $\Delta_{\gamma^x}$ and the operator $\widetilde{\Delta_\Omega^{x,\varepsilon}}$ 
		associated with $q_{\Omega^x_\varepsilon}$ on $J^x_\varepsilon(H^1_0(G^{\gamma^x}_N))$ can be  arbitrarily small as $\varepsilon \to 0+$. It remains to prove that the $N$-spectral difference between $\widetilde{\Delta_\Omega^{x,\varepsilon}}$ and $\Delta_\Omega^{x,\varepsilon}$ tends to zero as $\varepsilon \to 0+$. For this purpose, we analyze  the spaces 
		\[
		\mathcal H^\varepsilon_0:=J^x_\varepsilon(H^1_0(G^{\gamma^x}_N)) 
		\]
		and 
		\[
		\mathcal H^\varepsilon:=\{f\in H^1(\Omega_\varepsilon^x)|\ f=0  \text{ on }\{0\} \times B^{n-1}(\varepsilon)\text{ of $C_{k,\varepsilon}$}, \  \forall\ 1 \le k \le N\}
		\]
		which is the domain of the quadratic form associated with $\Delta_\Omega^{x,\varepsilon}$.
		
		Let $\mathcal H'_\varepsilon$ be the subspace of $\mathcal H^\varepsilon$ consisting of functions $f$ that satisfy:
		\begin{itemize}
			\item[(1)] On each cylinder  $C_{k,\varepsilon}$ and $C_{ij,\varepsilon}$,  the function $f$ is independent of $y'$ and depends linearly with respect to $y_1$.
			\item[(2)] On each $\partial \Omega_\varepsilon^x\cap\partial X_{l,\varepsilon}$, $\partial f/\partial\nu=0$ where $\nu$ is the  outward unit normal vector on $\partial \Omega_\varepsilon^x\cap\partial X_{l,\varepsilon}$.
			\item[(3)] On each  $X_{l,\varepsilon}$, $f$ is harmonic.
		\end{itemize}
		Denote $\mathcal H^{''}_\varepsilon=\mathcal H'_\varepsilon\cap \mathcal H^\varepsilon_0$, and define:
		\begin{itemize}
			\item $\mathcal H^{1,\varepsilon}_\infty=$  the orthogonal complement of $\mathcal H^{''}_\varepsilon$ in $\mathcal H'_\varepsilon$ under the inner product
		$
			\langle \nabla \cdot, \nabla \cdot\rangle_{L^2(\Omega^x_\varepsilon)},
		$
			\item $\mathcal H^{2,\varepsilon}_\infty=\{f\in \mathcal H^\varepsilon|\text{ For any cross section}  \{y_1\} \times B^{n-1}(\varepsilon) \text{ of the cylinder}$ 
			\\ \indent \qquad\;\qquad\qquad $C_{ij,\varepsilon}$ and $C_{k,\varepsilon}$, $\int_{B^{n-1}(\varepsilon)} f(y_1,y')\mathrm{d}y'=0\}.
			$
		\end{itemize} 
		It's not hard to check that 
		\[
		\begin{aligned} 
		\mathcal H^\varepsilon_0+\mathcal H^{1,\varepsilon}_\infty=\{f\in \mathcal H^\varepsilon|&\text{ On each cylinder of $\{C_{ij,\varepsilon},C_{k,\varepsilon}\}$, $f$ is independent of $y'$;}\\
		&\text{ On each $\partial \Omega_\varepsilon^x\cap\partial X_{l,\varepsilon}$, $\partial f/\partial\nu=0$;}\\
			&\text{ On each  $X_{l,\varepsilon}$, $f$ is harmonic}\}.
		\end{aligned} 
		\]
		Thus $\mathcal H^\varepsilon_0+\mathcal H^{1,\varepsilon}_\infty+\mathcal H^{2,\varepsilon}_\infty=\mathcal H^\varepsilon$. 
		
		For any $f\in \mathcal H_0^\varepsilon$ and $f_1\in \mathcal H^{1,\varepsilon}_\infty$, let $\tilde f\in\mathcal H^{''}_\varepsilon$ denote the function satisfying
		\[
		\tilde f=f\text{ on }X_{l,\varepsilon},\ \forall 1\le l\le N.
		\]
		Then
		\[
		\begin{aligned}
			 &\int_{\Omega_\varepsilon^x} \langle \nabla f,\nabla f_1\rangle \mathrm dV\\
			=&\sum_{1\le i<j\le N}\int_{C_{ij,\varepsilon}} \langle \nabla f,\nabla f_1\rangle \mathrm dV+\sum_{k=1}^N \int_{C_{k,\varepsilon}} \langle \nabla f,\nabla f_1\rangle \mathrm dV\\
			=&\sum_{1\le i<j\le N}\int_{\partial C_{ij,\varepsilon}} f\frac{\partial f_1}{\partial\nu} \mathrm d\sigma+\sum_{k=1}^N\int_{\partial C_{k,\varepsilon}} f\frac{\partial f_1}{\partial\nu} \mathrm d\sigma\\
			=&\sum_{1\le i<j\le N}\int_{\partial C_{ij,\varepsilon}} \tilde f\frac{\partial f_1}{\partial\nu} \mathrm d\sigma+\sum_{k=1}^N\int_{\partial C_{k,\varepsilon}} \tilde f\frac{\partial f_1}{\partial\nu} \mathrm d\sigma\\
			=&\int_{\Omega_\varepsilon^x} \langle \nabla \tilde f,\nabla f_1\rangle \mathrm dV\\
			=&0,
		\end{aligned}
		\]
		where the third equality uses the fact that $f_1$ only depends on the variable  $y_1$ on each cylinder $C_{ij,\varepsilon}$ and $C_{k,\varepsilon}$. For any $f\in \mathcal H_0^\varepsilon$ and $f_2\in \mathcal H^{2,\varepsilon}_\infty$, 
		\[
		\begin{aligned}
			 &\int_{\Omega_\varepsilon^x} \langle \nabla f,\nabla f_2\rangle \mathrm dV\\
			=&\sum_{1\le i<j\le N}\int_{C_{ij,\varepsilon}} \langle \nabla f,\nabla f_2\rangle \mathrm dV+\sum_{k=1}^N \int_{C_{k,\varepsilon}} \langle \nabla f,\nabla f_2\rangle \mathrm dV\\
			=&\sum_{1\le i<j\le N}\left(\int_{\partial C_{ij,\varepsilon}} f_2\frac{\partial f}{\partial\nu} \mathrm d\sigma+\int_{C_{ij,\varepsilon}}f_2\Delta f\mathrm dV\right)\\
			 &+\sum_{k=1}^N\left(\int_{\partial C_{k,\varepsilon}} f_2\frac{\partial f}{\partial\nu} \mathrm d\sigma+\int_{C_{k,\varepsilon}}f_2\Delta f\mathrm dV\right)
		\end{aligned}
		\]
		(the Laplacian above is positive operator). Since $f$ only depends on the variable  $y_1$ on each cylinder $C_{ij,\varepsilon}$ and $C_{k,\varepsilon}$, and for any cross section $\{y_1\} \times B^{n-1}(\varepsilon)$ of these cylinders , one has 
		\[\int_{B^{n-1}(\varepsilon)} f_2(y_1,y')\mathrm{d}y'=0,\]
		it follows that
		\[
		\int_{\Omega_\varepsilon^x} \langle \nabla f,\nabla f_2\rangle \mathrm dV=0.
		\]
		Thus, for any $f\in \mathcal H_0^\varepsilon, f_1+f_2\in \mathcal H^{1,\varepsilon}_\infty+\mathcal H^{2,\varepsilon}_\infty$, one has
		\[
		\int_{\Omega_\varepsilon^x} \langle \nabla f,\nabla (f_1+f_2)\rangle \mathrm dV=0.
		\]
		
		Next, we \textbf{claim} that if $\varepsilon$ is sufficiently small, there exists $C>0$ such that 
		\begin{equation}\label{lowbd}
			\int_{\Omega_\varepsilon^x} |\nabla f|^2\mathrm dV\ge C\varepsilon^{-1}\int_{\Omega_\varepsilon^x} f^2 \mathrm dV,\qquad \forall f\in \mathcal H^{1,\varepsilon}_\infty+\mathcal H^{2,\varepsilon}_\infty.
		\end{equation}
		Assuming this claim, the remainder of the proof of  Proposition \ref{graph-domain} follows by  standard arguments analogous  to the proof of Theorem IV.1 in \cite{CV2}, namely, one directly apply Theorem \ref{ct1}. Thus we omit the details for brevity.

		It remains to prove the claim. 		
		We first take $\varepsilon=1$. For any $f\in \mathcal H^{1,1}_\infty$, if 
		\[
		\sum_{l=1}^N  \int_{X_{l,1}} |\nabla f|^2\mathrm dV=0,
		\]
		then $f$ must be constant on each $X_{l,1}$. By definition of $\mathcal H^{1,1}_\infty$,  $f$ is zero. Since  $\dim(\mathcal H^{1,1}_\infty)=N^2-N$, there exists $C_1>0$ such that 
		\begin{equation}\label{C1H11}
		\int_{\Omega^x_1} |\nabla f|^2\mathrm dV\ge \sum_{l=1}^N \int_{X_{l,1}} |\nabla f|^2\mathrm dV\ge C_1\int_{\Omega^x_1} f^2 \mathrm dV,\qquad \forall f\in \mathcal H^{1,1}_\infty.
		\end{equation}
		For $\varepsilon<1$, a scaling argument yields 
		\begin{equation}\label{lowbd1}
			\int_{\Omega^x_\varepsilon} |\nabla f|^2\mathrm dV\ge \sum_{l=1}^N \int_{X_{l,\varepsilon}} |\nabla f|^2\mathrm dV\ge C_1 \varepsilon^{-1}\int_{\Omega^x_\varepsilon} f^2 \mathrm dV,\qquad \forall f\in \mathcal H^{1,\varepsilon}_\infty.
		\end{equation}
		
		The boundary $\partial W_{N+1}$ consists of $N$ disjoint unit balls $B_1,\cdots, B_N$ in $\mathbb R^{n-1}$, each associated with a boundary points of $S_{N+1}$. We now prove that the constant
		\begin{equation}\label{C2}
		C_2:=\inf\bigg\{\frac{\int_{W_{N+1}}|\nabla f|^2\mathrm dV}{\int_{W_{N+1}}f^2 \mathrm dV}\bigg|\ 0\neq f\in H^1(W_{N+1}),\ \int_{B_i} f\mathrm d\sigma=0,\ \forall 1\le i\le N\bigg\}
		\end{equation}
		is greater than 0. Suppose to the contrary that $C_2=0$, then there exists a sequence of functions $\{f_n\}_{n=1}^\infty$ in 
		\[
		\{f\in H^1(W_{N+1})|\ \int_{B_i} f\mathrm d\sigma=0,\ \forall 1\le i\le N\}\] 
		such that 
		\[
		\int_{W_{N+1}}f_n^2 \mathrm dV=1,\qquad \lim_{n\to\infty} \int_{W_{N+1}}|\nabla f_n|^2\mathrm dV=0.
		\]
		Thus, there exists a sub-sequence (still denoted by $\{f_n\}$) and a function $f_*\in H^1(W_{N+1})$ such that $f_n\rightharpoonup f_*$ weakly in $H^1(W_{N+1})$. This implies
		\begin{equation}\label{f*}
		\int_{W_{N+1}}f_*^2 \mathrm dV=1,\qquad \int_{B_i} f_*\mathrm d\sigma=0,\ \forall 1\le i\le N.
		\end{equation}
		Moreover, 
		\[
		\int_{W_{N+1}}|\nabla f_*|^2\mathrm dV\le \lim_{n\to\infty} \int_{W_{N+1}}|\nabla f_n|^2\mathrm dV=0.
		\]
		Hence $f_*$ is a constant function, which contradicts with \eqref{f*}. Therefore, $C_2>0$.
		Again, by scaling, for any $f\in \mathcal H^{2,\varepsilon}_\infty$, one has
		\begin{equation}\label{lowbd2}
			\begin{aligned}
			&\int_{\Omega^x_\varepsilon}|\nabla f|^2\mathrm dV\\
			=& \sum_{1\le i<j\le N}\int_{C_{ij,\varepsilon}} |\nabla f|^2\mathrm dV+\sum_{k=1}^N \int_{C_{k,\varepsilon}} |\nabla f|^2\mathrm dV+\sum_{l=1}^N \int_{X_{l,\varepsilon}} |\nabla f|^2\mathrm dV\\
			\ge& \varepsilon^{-2}\mu_1\cdot\big(\sum_{1\le i<j\le N}\int_{C_{ij,\varepsilon}}  f^2\mathrm dV+\sum_{k=1}^N \int_{C_{k,\varepsilon}}  f^2\mathrm dV\big)+\varepsilon^{-2}C_2\cdot \sum_{l=1}^N \int_{X_{l,\varepsilon}}  f^2\mathrm dV\\
			\ge& \varepsilon^{-2}\min(\mu_1,C_2)\int_{\Omega^x_\varepsilon} f^2\mathrm dV,
			\end{aligned} 
		\end{equation}
		where $\mu_1$ is the first nonzero eigenvalue of the Neumann Laplacian on $B^{n-1}$.

		For any function $f=f_1+f_2 \in \mathcal H^{1,\varepsilon}_\infty+\mathcal H^{2,\varepsilon}_\infty$, 
		\begin{equation}\label{111}
			\begin{aligned}
				&\int_{\Omega^x_\varepsilon}|\nabla(f_1+f_2)|^2\mathrm dV\\
				=&\int_{\Omega^x_\varepsilon}|\nabla f_1|^2 \mathrm dV+\int_{\Omega^x_\varepsilon}|\nabla f_2|^2\mathrm dV+2\int_{\Omega^x_\varepsilon}\langle \nabla f_1,\nabla f_2\rangle \mathrm dV\\
				=& \int_{\Omega^x_\varepsilon}|\nabla f_1|^2 \mathrm dV+\int_{\Omega^x_\varepsilon}|\nabla f_2|^2\mathrm dV+2\sum_{l=1}^N\int_{ X_{l,\varepsilon}} \langle \nabla f_1,\nabla f_2\rangle \mathrm dV.
			\end{aligned}
		\end{equation}
		Then, define
		\[
		\mathcal H^1_*:=\{f\in H^1(\cup_{l=1}^N X_{l,1})|\ f=\tilde f|_{\cup_{l=1}^N X_{l,1}}\text{ for some $\tilde f\in \mathcal H^{1,1}_\infty$}\},
		\]
		and
		\[
		\mathcal H^2_*:=\{f\in H^1(\cup_{l=1}^N X_{l,1})|\ f=\tilde f|_{\cup_{l=1}^N X_{l,1}}\text{ for some $\tilde f\in \mathcal H^{2,1}_\infty$}\}.
		\]
		By \eqref{C1H11}, 
		\begin{equation}\label{norm1}
		(\sum_{l=1}^N  \int_{X_{l,1}} |\nabla f|^2\mathrm dV)^{\frac 12}
		\end{equation}
		is a complete norm on $\mathcal H^1_*$. Moreover, since each $X_{l,1}$ is isometric to $W_{N+1}$ for all $1\le l\le N$, and the constant $C_2$ defined in \eqref{C2} is strictly positive, \eqref{norm1} also defines a complete norm on $\mathcal H^2_*$. 
		
		Next, suppose there exists $f_1\in \mathcal H^1_*$ and $f_2\in \mathcal H^2_*$ such that 
		\[
		(\sum_{l=1}^N  \int_{X_{l,1}} |\nabla (f_1+f_2)|^2\mathrm dV)^{\frac 12}=0.
		\] 
		Then $f_1+f_2$ must be constant on each $X_{l,1}$. By the definition of $\mathcal H^1_*$ and $\mathcal H^2_*$, one has
		\[
		f_2=0 \text{ on }\partial X_{l,1}\cap\left(\cup_{i,j}C_{ij,1}\cup_k C_{k,1}\right),\ \forall 1\le l\le N.
		\]
		Hence, $f_1$ is also constant on each $X_{l,1}$, and by the definition of $\mathcal H^1_*$, it follows that $f_1=0$. Consequently, $f_2=0$ as well. Therefore, \eqref{norm1} is also a norm on $\mathcal H^1_*\oplus \mathcal H^2_*$.
		
		For any $f\in \mathcal H^1_*$, let $\hat f$ be the best approximation of $f$ in $\mathcal H^2_*$ with respect to the inner product
		\[
		\langle \nabla \cdot, \nabla \cdot\rangle_{L^2(\cup_{l=1}^N X_{l,1})}.
		\]
		Then, since $\dim \mathcal H^1_*=\dim (\mathcal H^{1,1}_\infty)=N^2-N$ and $\mathcal H^1_*\cap \mathcal H^2_*=\{0\}$, one has
		\begin{equation}\label{C3}
			C_3:=\sup_{0\neq f\in \mathcal H^1_*}\frac{\sum_{l=1}^N  \int_{X_{l,1}} |\nabla \hat f|^2\mathrm dV}{\sum_{l=1}^N  \int_{X_{l,1}} |\nabla f|^2\mathrm dV}<1.
		\end{equation}
		According to \eqref{C3} and scaling, 
		\[
		\begin{aligned}
			\eqref{111}\ge& \int_{\Omega^x_\varepsilon}|\nabla f_1|^2 \mathrm dV+\int_{\Omega^x_\varepsilon}|\nabla f_2|^2\mathrm dV-2C_3^{\frac 12} (\sum_{l=1}^N\int_{ X_{l,\varepsilon}}|\nabla f_1|^2\mathrm dV)^{\frac 12}(\sum_{l=1}^N\int_{ X_{l,\varepsilon}}|\nabla f_2|^2\mathrm dV)^{\frac 12}\\
			\ge& \int_{\Omega^x_\varepsilon}|\nabla f_1|^2 \mathrm dV+\int_{\Omega^x_\varepsilon}|\nabla f_2|^2\mathrm dV-2C_3^{\frac 12} (\int_{\Omega^x_\varepsilon}|\nabla f_1|^2\mathrm dV)^{\frac 12}(\int_{\Omega^x_\varepsilon}|\nabla f_2|^2\mathrm dV)^{\frac 12}\\
			\ge& (1-C_3^{\frac 12})\int_{\Omega^x_\varepsilon}|\nabla f_1|^2 \mathrm dV+(1-C_3^{\frac 12})\int_{\Omega^x_\varepsilon}|\nabla f_2|^2\mathrm dV.
		\end{aligned}
		\]
	    Last, by \eqref{lowbd1} and \eqref{lowbd2}, one gets
	    \[
	    \begin{aligned}
	    	\eqref{111}\ge& (1-C_3^{\frac 12})\cdot\big(C_1\varepsilon^{-1}\int_{\Omega^x_\varepsilon} f_1^2\mathrm dV+\varepsilon^{-2}\min(\mu_1,C_2)\int_{\Omega^x_\varepsilon} f_2^2\mathrm dV\big)\\
	    	\ge& \frac 12 (1-C_3^{\frac 12})\varepsilon^{-1}\min(C_1,\varepsilon^{-1}\mu_1,\varepsilon^{-1} C_2)\int_{\Omega^x_\varepsilon} (f_1+f_2)^2\mathrm dV.
	    \end{aligned}
	    \]
	    This completes the proof of the claim.
	\end{proof}
	
	For the Dirichlet Laplacian $\Delta^\mathcal D(M,g)$, we denote its {\sl distinct} eigenvalues by 
	\[
	0< \Lambda_1^\mathcal D(M,g) < \Lambda_2^\mathcal D(M,g) < \Lambda_3^\mathcal D(M,g) < \cdots \to \infty,
	\]
	Now we prove the Dirichlet counterpart of Theorem \ref{premulRobin}: 
	
	\begin{theorem}\label{preDirimul}
		Let $M$ be a compact smooth $n$-manifold ($n \ge 3$) with nonempty boundary. Given any
		\begin{itemize}
			\item Multiplicity sequence $\{1, m_1,\cdots,m_s\}$ of positive integers, 
			\item Conformal class $[g]$ on $M$,
		\end{itemize}
		there exists a Riemannian metric $\tilde g \in [g]$ such that for all $1\le k \le s$, the eigenvalue  $\Lambda_{k+1}^\mathcal{D} (M, \tilde g)$ has multiplicity  $m_k$.
	\end{theorem}
	
	\begin{proof}
		Throughout the proof, we always assume that $x$ is sufficiently close to the origin of $\mathbb{R}^{\frac{N(N+1)}{2}}$, where $N=m_1+\cdots +m_s$. Since the dimension of $M$ is at least 3, there exists a metric $g_x \in [g]$ such that the metric graph $G_N^{\gamma^x}$ can be isometrically embedded in $(M,g_x)$ without edge crossing, with the boundary of $G_N^{\gamma^x}$ lying on $\partial M$ and intersecting $\partial M$ perpendicularly. This can be achieved in two steps: First, embed $G_N$ topologically into  $(M,g)$ without edge crossings, such that the boundary of $G_N$ lies on $\partial M$ and intersect $\partial M$ perpendicular. Second, conformally adjust the metric $g$ in a small neighborhood around the midpoint of each edge of $G_N$ to match the prescribed edge length. Since the length of each edge in $G_N^{\gamma^x}$ depends smoothly on $x$, one can ensure that $g_x$ is $C^1$-continuous with respect to the parameter $x$.
		
     	Next, following the method in the proof of \cite[Theorem 1.2]{PJSte}, for a sufficiently small $\varepsilon > 0$, one can deform the metric $g_x$ (in a non-conformal way) to a metric $g_x^\varepsilon$ such that $G_N^{\gamma^x}$ remains isometrically embedded in $(M, g_x^\varepsilon)$, and such that the image of $G_N^{\gamma^x}$ admits a neighborhood in $(M, g_x^\varepsilon)$ that is isometric to the domain $\Omega^x_\varepsilon$. For simplicity we also use $\Omega^x_\varepsilon$ to denote that neighborhood.  The existence of such a metric follows from the fact that there is a canonical map from a tubular neighborhood of a line segment in   $(M,g_x)$ (where the metric in the neighborhood is nearly Euclidean) to a tubular neighborhood of the zero section of the normal bundle of that line.
		One can further ensure that
		\[
		\frac {1} {\tau_\varepsilon} g^\varepsilon_x \le g_x \le \tau_\varepsilon g_x^\varepsilon
		\]
		for a family of real numbers $\tau_\varepsilon > 1$ independent of $x$ such that $\tau_\varepsilon \to 1$ as $\varepsilon \to 0$. 
		
		Now we demonstrate that the $N$-spectral difference between $\bar \Delta_{\gamma^x}$ on $G_N^{\gamma^x}$ and the Dirichlet Laplacian on $M$ with a certain metric can be arbitrarily small, through the following three steps: 
		\begin{itemize}
			\item By Proposition \ref{graph-domain}, the $N$-spectral difference between $\Delta_\Omega^{x,\varepsilon}$ of $\Omega^x_\varepsilon$ and $\bar \Delta_{\gamma^x}$ on $G_N^{\gamma^x}$ tends to zero as $\varepsilon \to 0$.
			\item According to \cite[Lemma 5.1]{HW}, there exists a metric $\widetilde {g^\varepsilon_x}$ on $M$ satisfying
			\[
			\widetilde {g^\varepsilon_x}|_{\Omega^x_\varepsilon}=g^\varepsilon_x,
			\]
			such that the $N$ spectral difference between $\Delta^{x,\varepsilon}_\Omega$ on $\Omega^x_\varepsilon$ and the Dirichlet Laplacian $\Delta ^\mathcal D(M,\widetilde {g^\varepsilon_x})$ can be made arbitrarily small as $\varepsilon \to 0$. Moreover, $\widetilde {g^\varepsilon_x}$ can be chosen in the conformal class of $g^\varepsilon_x$, as demonstrated in the proof of \cite[Lemma 5.1]{HW}. 
			Thus, there exists a family of uniformly bounded smooth functions $\{f^h_{x,\varepsilon}\}_{h>0}$ on $M$ such that the $N$-spectral difference between $\Delta^{x,\varepsilon}_\Omega$ and $\Delta^\mathcal D(M, f^h_{x,\varepsilon}\cdot g^\varepsilon_x)$ tends to zero as $h \to 0$. 
			\item Since $\tau_\varepsilon \to 1$ as $\varepsilon \to 0$, the $N$-spectral difference between $\Delta^\mathcal D(M, f^h_{x,\varepsilon}\cdot g^\varepsilon_x)$ and $\Delta^\mathcal D(M, f^h_{x,\varepsilon} \cdot g_x)$ tends to zero as $\varepsilon \to 0$.
		\end{itemize}
		So the $N$-spectral difference between $\Delta^\mathcal D(M, f^h_{x,\varepsilon} \cdot g_x)$ and $\bar \Delta_{\gamma^x}$ tends to zero as   $h, \varepsilon \to 0$. Moreover, by choosing the functions $\{f^h_{x,\varepsilon} \}$ to depend $C^1$-continuously on the parameters $x$, and uniformly in  $\varepsilon$ and $h$, one can make this convergence locally uniform on $x$ .
		
		Let $B \subset \mathbb{R}^{\frac{N(N+1)}{2}}$ be a ball center at the origin with sufficiently small radius. For $x\in B$ and sufficiently small $\varepsilon$ and $h$, the following space is well defined:
		\[
		E_x^{\varepsilon,h}:=\Span\{\varphi_2(x,\varepsilon,h),\cdots,\varphi_N(x,\varepsilon,h)\},
		\]
		where $\{\varphi_2(x,\varepsilon,h),\cdots,\varphi_N(x,\varepsilon,h)\}$ are the eigenfunctions associated with Dirichlet eigenvalues $\{\lambda_2^\mathcal D(M, f^h_{x,\varepsilon} \cdot g_x),\cdots, \lambda_N^\mathcal D(M, f^h_{x,\varepsilon} \cdot g_x)\}$. Following the approach in \S\ref{sm}, one can construct an isometry 
		\[
		I_x^{\varepsilon,h}:E_0\to E_x^{\varepsilon,h}
		\]
		(The definition of $E_0$ appears before Proposition \ref{submer}), which induces the map
		\[
		\Phi_h^\varepsilon: B\to Q(E_0),\qquad x\mapsto q^\mathcal D_{x,\varepsilon,h}\circ I_x^{\varepsilon,h}
		\]
		where $q^\mathcal D_{x,\varepsilon,h}$ is the quadratic form associated with $\Delta^\mathcal D(M, f^h_{x,\varepsilon} \cdot g_x)$, and $Q(E_0)$ is the space of quadratic forms on $E_0$.
		
		Since the $N$-spectral difference between $\Delta^\mathcal D(M, f^h_{x,\varepsilon} \cdot g_x)$ and $\bar \Delta_{\gamma^x}$ tends to zero as $h,\varepsilon \to 0$, the map $\Phi_h^\varepsilon$ converges to the map $F$ constructed in Proposition \ref{submer}. 
		According to Proposition \ref{submer}, the metric $\gamma^0$ on $G_N$ is stable. Therefore, for sufficiently small $h$ and $\varepsilon$, $\Lambda_2\cdot \mathrm{Id}$ is an interior point of $\operatorname{Im}(\Phi_h^\varepsilon)$. This completes the proof of Theorem \ref{preDirimul}.
	\end{proof}
	
	\begin{proof}[Proof of Theorem \ref{premulRobin}]
	By Proposition \ref{Diri-Robin}, the family of metrics
	\[
	\{f^h_{x,\varepsilon} \cdot g_x\}_{x\in B}
	\]
	is a stable family for the Robin Laplacian on $M$ with the constant boundary Robin parameter $\rho =C$,  when $C$ is sufficiently large. Now Theorem  \ref{premulRobin} follows from Proposition \ref{DtoR}.	
	\end{proof}
	
	In the remaining of this section  we prove Theorem \ref{prediffeigen}.
	
	\begin{proof}[Proof of Theorem \ref{prediffeigen}]
		Propositions \ref{Diri-Robin} and \ref{DtoR} allow us to reduce the Robin boundary problem to its Dirichlet counterpart. We therefore focus only on the corresponding Dirichlet boundary problem in what follows.
		
		Since 
		\[
		\lim_{N\to\infty}\frac{\arccos\frac{N-1}{N}}{\arccos \frac{-1}{N}}=0,
		\]
		one can choose a sufficiently large $N$ such that 
		\[
		{a_1}\cdot \frac{(\arccos \frac{-1}{N})^2}{(\arccos\frac{N-1}{N})^2}> a_m+1.
		\]
		By scaling, for any $c>0$, one has
		\[
		\spec(\Delta_{c\gamma^0})=\{c^{-2}\lambda|\ \lambda\in \spec(\Delta_{\gamma^0})\}
		\]
		where $\spec(\Delta_{c\gamma^0})$ and $\spec(\Delta_{\gamma^0})$ denote the eigenvalues of $\Delta_{c\gamma^0}$ and $\Delta_{\gamma^0}$, respectively. Define the graph $\widetilde G$ as the disjoint union of the following $m$ graphs
		\[
		\big\{(G_N, \frac{\arccos\frac{N-1}{N}}{\sqrt{a_i}}\gamma^0)\big\}_{i=1}^m,
		\]
		then the first $m$ eigenvalues of the Dirichlet Laplacian on $\widetilde G$ are precisely $a_1,\cdots,a_m$, while the $(m+1)$-th eigenvalue of $\widetilde G$ is greater than $a_m+1$.
		
		Next, we claim that for a sufficiently large number $\alpha>0$, there exists a bounded domain $(\Omega,\alpha\cdot g)\subset (M,\alpha\cdot g)$ such that
		\begin{itemize}
			\item[(1)] Let $\lambda_1(\Omega)$ be the first eigenvalue  of the Laplacian on $(\Omega,\alpha\cdot g)$ with Dirichlet boundary condition on $\partial M\cap \bar \Omega$ and Neumann boundary condition on $\partial \Omega\cap \inte(M)$, then $\lambda_1(\Omega)>a_m+1$,
			\item[(2)] $\vol(\Omega,\alpha\cdot g)=V$.
		\end{itemize}
		
		To prove this claim, we first note that for fixed $L>0$ and $\varepsilon>0$, the first eigenvalue of the Laplacian on Euclidean domain $B^{n-1}(\varepsilon)\times [0,L]$, with Dirichlet boundary condition on $B^{n-1}(\varepsilon)\times\{0,L\}$ and Neumann boundary condition on $\partial B^{n-1}(\varepsilon)\times [0,L]$, is $\frac{\pi^2}{L^2}$. Let $l$ be any simple curve in $M$ that intersects $\partial M$ perpendicularly at its endpoints. For sufficiently small $\varepsilon > 0$, let $T_\varepsilon(l)$ denote the $\varepsilon$-tubular neighborhood of $l$. By the Tubular Neighborhood Theorem, there exists a diffeomorphism from $T_\varepsilon(l)$ to $B^{n-1}(\varepsilon) \times [0, \Length(l)]$. Since $\varepsilon$ is small, this diffeomorphism is almost an isometry. Consequently, the first eigenvalue of the Laplacian on $T_\varepsilon(l)$, with Dirichlet boundary condition on $\partial M\cap T_\varepsilon(l)$ and Neumann boundary condition on $\partial T_\varepsilon(l)\cap \inte (M)$, is nearly $\pi^2/\Length(l)^2$. So we may take $(\Omega,\alpha\cdot g)$ to be the disjoint union of many tubular neighborhoods of   simple curves in $M$ of length less than $\pi/\sqrt{a_m+1}$, see the graph below:\\
		\includegraphics[scale=0.35]{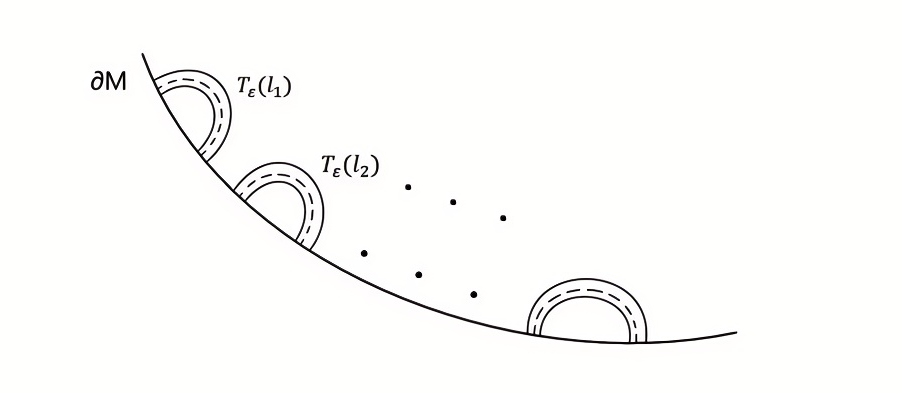}
		
		Let 
		\[X:=(X_1,\cdots,X_m)\in \mathbb R^{\frac{N(N+1)}{2}\times m}\]
		be a point sufficiently close to the origin. Then, as in the proof of Theorem \ref{preDirimul}, there exists a metric $g(X)\in [g]$ such that the following $m$ graphs
		\[
		\bigg\{G_N^{\frac{\arccos\frac{N-1}{N}}{\sqrt{a_i}}\gamma^{X_i}}=:G_N(X_i)\bigg\}_{i=1}^m
		\]
		can be isometrically embedded in $(M, g(X))$ with no edge intersections, and each $G_N(X_i)$ intersects $\partial M$ perpendicularly along its boundary. Moreover, the domain $\Omega$ can be chosen to be disjoint from the embedded graphs $\{G_N(X_i)\}_{i=1}^m$, and  $g(X)=\alpha\cdot g$ on $\Omega$.
		
		Next, as in the proof of Theorem \ref{preDirimul}, we perturb the metric $g(X)$ non-conformally to obtain a new metric $g^{\varepsilon_1}(X)$ such that:
		\begin{enumerate}
			\item The $m$ scaled domains 
			\[
			\big\{ \frac{\arccos\frac{N-1}{N}}{\sqrt{a_i}}\cdot \Omega_{\varepsilon_1}^{X_i} \big\}_{i=1}^m
			\]
			can be isometrically embedded in $(M,g^{\varepsilon_1}(X) )$ (where $\Omega_{\varepsilon_1}^{X_i}$ is defined at the beginning of this section);
			\item There exists a family of constants $\tau_{\varepsilon_1} > 1$ with $\tau_{\varepsilon_1} \to 1$ as $\varepsilon_1 \to 0$, such that
			\[
			\frac {1} {\tau_{\varepsilon_1}} g^{\varepsilon_1}(X) \le g(X) \le \tau_{\varepsilon_1} g^{\varepsilon_1}(X);
			\]
			\item Let $\tilde{\Omega}_{\varepsilon_1}^{X_i}$ denote the image of $\frac{\arccos\frac{N-1}{N}}{\sqrt{a_i}} \cdot \Omega_{\varepsilon_1}^{X_i}$ under the embedding. Then
			\[
			\tilde{\Omega}^{X_i}_\varepsilon\cap \tilde{\Omega}^{X_j}_\varepsilon=\emptyset,\qquad \forall\  1\le i\neq j\le m,
			\]
			and 
			\[
			\Omega\cap \tilde\Omega^{X_i}_\varepsilon=\emptyset,\qquad \forall \ 1\le i\le m;
			\]
			\item $g^{\varepsilon_1}(X)=\alpha\cdot g$ on $\Omega$.
		\end{enumerate}
		
		Next, as in the proof of Theorem \ref{preDirimul}, and again by \cite[Lemma 5.1]{HW}, there exists a family of uniformly bounded smooth functions $\{f^h_{X,\varepsilon_1}\}_{h>0}$ on $M$ such that 
		\begin{enumerate}
			\item the $N$-spectral difference between $\Delta^\mathcal D(M, f^h_{X,\varepsilon_1}g^{\varepsilon_1}(X))$ and the Dirichlet Laplacian $\Delta(\widetilde G)$ on the graph $\widetilde G$ tends to zero as $h,\varepsilon_1\to 0$;
			\item the $N$-spectral difference between 
			\[\Delta^\mathcal D(M, f^h_{X,\varepsilon_1}g^{\varepsilon_1}(X))\text{ and }\Delta^\mathcal D(M, f^h_{X,\varepsilon_1}g(X))\]
			tends to zero as $\varepsilon_1\to 0$;
			\item $f^h_{X,\varepsilon_1}=1$ on $\Omega$, and
			\begin{equation}\label{convol}
			\lim_{h,\varepsilon_1\to 0}\vol(M,f^h_{X,\varepsilon_1}g(X))=V.
			\end{equation}
		\end{enumerate}
		Note that one can apply \cite[Lemma 5.1]{HW} since the first eigenvalue of $(\Omega,\alpha\cdot g)$ and the second eigenvalue of each $\Omega^{X_i}_{\varepsilon_1}$ are greater than $a_m+1$ for sufficiently small ${\varepsilon_1}$.
		
		Let $B \subset \mathbb{R}^{\frac{N(N+1)}{2}\times m}$ be a ball center at the origin with sufficiently small radius. Consider the map
		\[
		\Phi^{\varepsilon_1}_h: B\to\mathbb R^m,\qquad X\mapsto (\lambda_1^\mathcal D(M,f^h_{X,\varepsilon_1}g(X)),\cdots, \lambda_m^\mathcal D(M,f^h_{X,\varepsilon_1}g(X))).
		\]
		By Proposition \ref{simple stable}, 
		\[
		\{(M,f^h_{X,\varepsilon_1}g(X))\}_{X\in B}
		\]
		forms a stable family of metrics, and the point $(a_1,\cdots,a_m)$ lies in the interior of  $\operatorname{Im}(\Phi^{\varepsilon_1}_h)$. 
		
		Finally, to simultaneously prescribe the volume of $M$, we choose $\varepsilon_1$ sufficiently small. By \eqref{convol}, one can normalize the volume of $(M,f^h_{X,\varepsilon_1}g(X))$ to $V$ within the fixed conformal class $[g]$, without affecting the stability of $\{(M,f^h_{X,\varepsilon_1}g(X))\}_{X\in B}$. This completes the proof.
	\end{proof}

	\appendix

	\section{}\label{app}
	
	In this appendix we collect some necessary backgrounds on spectral convergence and stable metrics that are needed.
		
	For any closed positive quadratic form $q$ with domain $D(q)$, there exists  a unique positive self-adjoint operator $A$ (whose domain is dense in $D(q)$) so that $q$ is the quadratic form associated with $A$ (c.f.  Theorem VIII.15 in \cite{RS}). When $A$ possesses discrete spectrum $\{\lambda_i\}_{i=1}^\infty$ satisfying $\lambda_{N+1}>\lambda_N$ for some $N$,   the subspace of $D(q)$ spanned by the first $N$ eigenfunctions of $A$ is well-defined, and is called the \textbf{$N$-eigenspace} of $q$ (resp. $A$).  Unless otherwise specified, all quadratic forms considered in this appendix are assumed to have discrete spectrum.
	
	\setcounter{theorem}{0} 
	\renewcommand{\thetheorem}{A.\arabic{theorem}}
	\subsection{Spectral convergence}\label{Aoe}

	\noindent
	
	Given integer $N>0$, let $E_0$ and $E_1$ be two $N$-dimensional subspaces of a real Hilbert space $(\mathcal{H},\langle\ ,\ \rangle)$. Endow each $E_i$ with an inner product
	\begin{equation*}
		\langle x,y \rangle_i=\langle A_ix,y \rangle,\ x,y\in E_i, \qquad i=0,1,
	\end{equation*}
	where $A_i: E_i \to E_i$ is a strictly positive operator. Moreover, suppose $E_1$ is \textbf{close} to $E_0$, in the sense that $E_1$ is the graph of a  bounded linear map $B\in\mathcal{L}(E_0,E_0^\perp)$. Let $\mathcal{B}=I+B\in\mathcal{L}({E_0,E_1})$ and consider the isometry
	\begin{equation}\label{iso}
		U_{E_0,E_1}=A_1^{-\frac{1}{2}}\mathcal{B}(\mathcal{B}^*\mathcal{B})^{-\frac{1}{2}}A_0^{\frac{1}{2}}
	\end{equation}
	from $(E_0,\langle\ ,\ \rangle_0)$ to $(E_1,\langle\ ,\ \rangle_1)$.
	The following definition is taken from \cite{CV2}:
	\begin{defn}\label{sd}
		For positive quadratic forms  $q_i$ on $(E_i,\langle\ ,\ \rangle_i)$, $i=0,1$, if
		\[\|q_1\circ U_{E_0,E_1}-q_0\|_\infty \leq\varepsilon,\]
		then we say $(E_0,\langle\ ,\ \rangle_0,q_0)$ and $(E_1,\langle\ ,\ \rangle_1,q_1)$ are $\varepsilon$-\textbf{close}. In particular, if $E_i$ is the $N$-eigenspace of a quadratic form $Q_i$  with $q_i=Q_i|_{E_i}$, and if $(E_0,\langle\ ,\ \rangle_0,q_0)$ and $(E_1,\langle\ ,\ \rangle_1,q_1)$ are  $\varepsilon$-close, then we say $Q_0$ and $Q_1$ have   \textbf{$N$-spectral difference $\leq\varepsilon$}.
	\end{defn}

	A useful criterion for $\varepsilon$-closeness was introduced by Y. Colin de Verdi\`ere:
	\begin{lemma}[\cite{CV2}, Critere I.3]\label{crit}
		For any $\varepsilon>0$, there exists $M$ and $\alpha_i>0$ $(1\leq i\leq 5)$ depending on $\varepsilon$, such that if the following conditions hold:
		\begin{itemize}
			\item  $\|q_1\|\leq M$,
			\item  $\|A_0-I\|\leq\alpha_1$, $\|A_1-I\|\leq\alpha_2$ and   $\|B\|\leq\alpha_3$, 
			\item $\max_{1\leq j\leq N}|\lambda_j(q_1)-\lambda_j(q_0)|\leq\alpha_4$, 
			\item $q_1(x+Bx)\geq q_0(x)-\alpha_5|x|^2$ for all $x \in E_0$,
		\end{itemize} then $(E_0,\langle\ ,\ \rangle_0,q_0)$ and $(E_1,\langle\ ,\ \rangle_1,q_1)$ are $\varepsilon$-close.
	\end{lemma}
	
	\par Let $Q$ be a closed positive quadratic form on $\mathcal{H}$ with discrete spectrum. We say $Q$ satisfies the spectral gap hypothesis ($\ast$) if its eigenvalues fulfill
	\[\label{hoe}
	\tag{$\ast$}\ \lambda_1\leq\cdots\leq\lambda_N\leq\lambda_N+\delta \le \lambda_{N+1}\leq M ,
	\]
	where $M,N,\delta$ will be fixed once and for all. Then we have
	
	\begin{theorem}\cite[Theorem 2.3]{HW}
		\label{ct1}
		Let $Q$ be a positive quadratic form on a Hilbert space $\mathcal{H}$ whose domain admits a  $Q$-orthogonal decomposition $D(Q)=\mathcal{H}_0\oplus\mathcal{H}_\infty$. Suppose $Q_0=Q|_{\mathcal{H}_0}$ (whose eigenvalues are denoted by $\{\mu_i\}$) verifies the hypothesis (\ref{hoe}). 
		{Let $\tilde \delta$ be the minimum of the gap between distinct eigenvalues from $\mu_1$ to $\mu_{N+1}$,}
		then for any $\varepsilon>0$, there exists a constant $C>0$ (depending on $\tilde \delta,M,N$ and $\varepsilon$) such that if $Q(x)\geq C|x|^2$ for all $x\in\mathcal{H}_\infty$, the quadratic forms $Q_0$ and $Q$ have a $N$-spectral difference $\leq\varepsilon$.
	\end{theorem}

	We also need Remark 18 in \cite{PJ}, which is a variant of Theorem I.8 in \cite{CV2}:
	\begin{theorem}[\cite{CV2}, \cite{PJ}]\label{ct2}
		Let $Q$ be a positive quadratic form satisfying the hypothesis (\ref{hoe}), and   $\langle\ ,\ \rangle$   an inner product on $D(Q)$. Consider sequences $\langle\ ,\ \rangle_n$ of inner products and $Q_n$ of quadratic forms on $D(Q)$ satisfying
		\begin{enumerate}
			\item for any $x\in D(Q)$,
			\[\lim_{n\to\infty}|x|_n=|x|, \quad \lim_{n\to\infty}Q_n(x)=Q(x), \quad Q(x)\leq Q_n(x),\]
			\item there exists $C_1,C_2$ and $\varepsilon_n\to 0+$ such that \[C_1|x|\leq|x|_n\leq C_2|x|+\varepsilon_nQ(x)^{\frac{1}{2}}, \quad \forall x\in D(Q).\]
		\end{enumerate}
		Then for any $\varepsilon>0$, there exists  $K=K(\delta,M,N,\varepsilon) \in \mathbb N$ such that $Q$ and $Q_n$ have an $N$-spectral difference $\leq\varepsilon$ for all $n>K$.
	\end{theorem}
	
	In most applications, we take $Q$ to be the  Laplacian on a Riemannian manifold (with specific boundary condition), and take  $\langle\ ,\ \rangle$ to be the $L^2$-inner product.

	\subsection{Stable metrics} \label{sm}
	
	\noindent
	
	The isomorphism (\ref{iso}) plays a crucial role in our argument. To simplify the description, we introduce the concept of stable metrics (following \cite{JL96}).
	
	\par Let $M,\ M'$ be two compact manifolds with piecewise smooth boundaries and let $B$ be a closed ball in $\mathbb{R}^m$.  Let $\mathcal{M}(M)$ and $\mathcal{M}(M')$ be the space of  Riemannian metrics on $M$ and $M'$ and consider two continuous families of Riemannian metrics 
	\[f:B\to\mathcal{M}(M) \ \text{and}\ F:B\to\mathcal{M}(M').\]
	For any $g\in\mathcal{M}(M)$, let $E_N(g)$ be the $N$-eigenspace of the Laplacian of $g$ with certain boundary conditions and $\mathcal{Q}_N(g)$ be the space of quadratic forms on $E_N(g)$. Fix a Riemannian metric $g_0$ on $M$ and consider the $L^2$-isometry
	\[i_p: (E_N(g_0), \langle , \rangle_{g_0})\to (E_N(f(p)), \langle , \rangle_{f(p)})\]
	constructed via (\ref{iso}) (which depends continuously on $p\in B$). Let  
	\[ U_p:H^1(M,f(p))\to H^1(M',F(p))\]
	be an  injective map
	using which we can pull-back the inner product on $U_p(E_N(f(p)))$ to $E_N(f(p))$. Composing with  (\ref{iso}), we get an $L^2$-isometry
	\[I_p^1:(E_N(f(p)), \langle , \rangle_{f(p)}) \stackrel{(\ref{iso})}{\to} (E_N(f(p)), U_p^*(\langle , \rangle_{F(p)})) \to  (U_p(E_N(f(p))), \langle , \rangle_{F(p)}),\]
	Also (\ref{iso}) gives rise to another $L^2$-isometry
	\[I_p^2:(U_p(E_N(f(p))), \langle , \rangle_{F(p)})\to (E_N(F(p)), \langle , \rangle_{F(p)}).\]
	Composing the above three isometries, we get an $L^2$-isometry
	\[I_p = I_p^2 \circ I_p^1 \circ i_p: (E_N(g_0), \langle , \rangle_{g_0}) \to (E_N(F(p)), \langle , \rangle_{F(p)}).\]
	
	Using the isometries $i_p$ and $I_p$, the quadratic forms associated with the Laplacian on $(M, f(p))$ and $(M', F(p))$ with domains $E_N(f(p))$ and $E_N(F(p))$, respectively, can be transformed into quadratic forms on $E_N(g_0)$. Consequently, we obtain two mappings from $B$ to $\mathcal{Q}_N(g_0)$:
	\[
	\Phi(f): B \to \mathcal{Q}_N(g_0),
	\]
	\[
	\Phi(F): B \to \mathcal{Q}_N(g_0).
	\]
	For example, if we consider the Laplacian with Robin boundary condition, then for any $p\in B$ and $\varphi\in E_N(g_0)$,
	\[
	[\Phi(f)(p)](\varphi)=\int_M |\nabla_{f(p)}i_p(\varphi)|^2 \mathrm d V_{f(p)}+\int_{\partial M} \rho\cdot i_p(\varphi)^2 \mathrm d \sigma_{f(p)}.
	\]
	
	Finally we write down the definition of stable metrics. For simplicity, suppose the ball  $B$ is centered at the origin.  
	\begin{defn}[\cite{JL96} Definition 2.1]
		Let $f:B\to\mathcal{M}(M)$ be a continuous family of Riemannian metrics  on $M$ with $f(0)=g_0$. We say $f$ is  a \textbf{stable} family  around $g_0$, if there exists an $\varepsilon=\varepsilon(f,g_0)>0$ such that
		for any continuous  family $F:B\to\mathcal{M}(M')$  of Riemannian metrics on $M'$  with \[\|\Phi(F)-\Phi(f)\|_{C^0(B)}<\varepsilon,\] there exists a point $p\in\mathrm{int}(B)$ with
		\[\Phi(F)(p)=\Phi(f)(0).\]
		We also call $g_0$ a \textbf{stable metric}, and say $F$ is \textbf{spectrally near} to $f$.
	\end{defn}
	
	\section*{Acknowledgment}
	
	\noindent {\bf Funding} The authors are partially supported by  National Key R and D Program of China 2020YFA0713100, and by NSFC no. 12171446. The first author is supported by the China Postdoctoral Science Foundation 2024M761591.


\begin{thebibliography}{99}
		
		\bibitem{Be18}
		A. Berdnikov: Bounds on multiplicities of Laplace operator eigenvalues on surfaces. \emph{Journal of Spectral Theory}, 2018, 8(2): 541-554.
		
		\bibitem{BFK}
		D. Bucur, P. Freitas and J. Kennedy: The Robin problem. Chapter 4 in {\sl A. Henrot Ed: Shape optimization and spectral theory}. \emph{De Gruyter Open}, 2017.
		
		\bibitem{C-C}
		B. Colbois and Y. Colin de Verdi\`{e}re: Sur la multiplicit\'{e} de la premi\`{e}re valeur propre d'une surface de Riemann \`{a} courbure constante. \emph{Commentarii Mathematici Helvetici}, 1988, 63: 194-208.
		
		\bibitem{CV2}
		Y. Colin de Verdi\`{e}re: Sur la multiplicit\'{e} de la premi\`{e}re valeur propre non nulle du Laplacien. \emph{Commentarii Mathematici Helvetici}, 1986, 61(1): 254-270.
		
		\bibitem{CV} 
		Y. Colin de Verdière: Construction de laplaciens dont une partie finie du spectre est donn\'ee \emph{Annales scientifiques de l'\'ecole normale sup\'erieure}, 1987, 20(4): 599-615.
		
		\bibitem{Da05} 
		M. Dahl: Prescribing eigenvalues of the Dirac operator. \emph{manuscripta mathematica}, 2005, 118(2): 191-199.
		
		\bibitem{EI} 
		A. El Soufi, S. Ilias: Immersions minimales, premiere valeur propre du laplacien et volume conforme. \emph{Mathematische Annalen}, 1986, 275: 257-267.
		
		\bibitem{Gu04} 
		P. Guerini: Prescription du spectre du laplacien de Hodge–de Rham. \emph{Annales Scientifiques de l'\'Ecole Normale Supérieure}, 2004, 37(2): 270-303.
		
		\bibitem{Has}
		A. Hassannezhad: Conformal upper bounds for the eigenvalues of the Laplacian and Steklov problem. \emph{Journal of Functional analysis}, 2011, 261(12): 3419-3436.
		
		\bibitem{HKP} 
		A. Hassannezhad, G. Kokarev, I. Polterovich: Eigenvalue inequalities on Riemannian manifolds with a lower Ricci curvature bound. \emph{Journal of Spectral Theory}, 2016, 6(4): 807-835.
		
		\bibitem{HW}
		X. He and Z. Wang: Riemannian metrics with prescribed volume and finite parts of Dirichlet spectrum. \emph{Calculus of Variations and Partial Differential Equations}, 2024, 63: 78.
		
		\bibitem{H}
		X. He: Prescription of finite Dirichlet eigenvalues and area on surface with boundary. \emph{Journal of Geometry and Physics}, 2024, 198: 105100.
		
		\bibitem{PJcon}
		P. Jammes: Prescription du spectre du laplacien de Hodge–de Rham dans une classe conforme. \emph{Commentarii Mathematici Helvetici}, 2008, 83(3): 521-537.
		
		\bibitem{PJ09} 
		P. Jammes: Sur la multiplicité des valeurs propres d’une vari\'et\'e compacte. \emph{S\'eminaire de th\'eorie spectrale et g\'eom\'etrie}, 2009, 26: 1-11.
		
		\bibitem{PJJGA}
		P. Jammes: Construction de valeurs propres doubles du laplacien de Hodge-de Rham. \emph{Journal of Geometric Analysis}, 2009, 19: 643-654.
		
		\bibitem{PJ}
		P. Jammes: Prescription de la multiplicit\'{e} des valeurs propres du laplacien de Hodge-de Rham. \emph{Commentarii Mathematici Helvetici}, 2011, 86(4): 967-984.
		
		\bibitem{PJ12}
		P. Jammes: Sur la multiplicit\'e des valeurs propres du laplacien de Witten. \emph{Transactions of the American Mathematical Society}, 2012, 364(6): 2825-2845.
		
		\bibitem{PJSte}
		P. Jammes: Prescription du spectre de Steklov dans une classe conforme. \emph{Analysis \& PDE}, 2014, 7(3): 529-550.
		
		\bibitem{Kor93} 
		N. Korevaar: Upper bounds for eigenvalues of conformal metrics. \emph{Journal of Differential Geometry}, 1993, 37(1): 73-93.
		
		\bibitem{LY}
		P. Li, S T. Yau: A new conformal invariant and its applications to the Willmore conjecture and the first eigenvalue of compact surfaces. \emph{Inventiones mathematicae}, 1982, 69(2): 269-291.
		
		\bibitem{JL96}
		J. Lohkamp: Discontinuity of geometric expansions. \emph{Commentarii Mathematici Helvetici}, 1996, 71(1): 213-228.
		
		\bibitem{RS}
		M. Reed and B. Simon: Methods of modern mathematical physics: Functional analysis; Rev. ed. \emph{Academic press}, 1980.
		
		
		
		
	\end{thebibliography}
\end{document}